%% file: main.tex
\documentclass{amsart}

\usepackage{algorithmic}
\usepackage{amsmath}
\usepackage{amsthm}
\usepackage[dvipsnames,usenames]{color}
\usepackage[colorlinks=true, urlcolor=NavyBlue, linkcolor=NavyBlue, citecolor=NavyBlue]{hyperref}
\usepackage{etoolbox}
\usepackage{mathtools}
\usepackage{graphicx}

\usepackage{scalerel}

\newtheorem{lemma}{Lemma}[section]
\newtheorem{corollary}[lemma]{Corollary}
\newtheorem{proposition}[lemma]{Proposition}

\newtheorem*{theorem*}{Theorem}

\newtheorem*{conjecture_main}{Conjecture (Boyer-Gordon-Watson)}
\theoremstyle{definition}
\newtheorem{definition}[lemma]{Definition}
\newtheorem{fact}[lemma]{Fact}
\theoremstyle{remark}
\newtheorem*{remark}{Remark}

\numberwithin{equation}{section}

\DeclarePairedDelimiterX\Set[2]{\lbrace}{\rbrace}%
 { #1 \mid{} #2 }

\DeclareRobustCommand{\numToLet}[1]{\ifstrequal{#1}{1}{a}
{\ifstrequal{#1}{2}{b}
{\ifstrequal{#1}{3}{c}
{\ifstrequal{#1}{4}{d}
{\ifstrequal{#1}{5}{e}
{\ifstrequal{#1}{6}{f}
{\ifstrequal{#1}{7}{g}
{\ifstrequal{#1}{8}{h}{out of bounds}
}}}}}}}}

\DeclareRobustCommand{\case}[3]{\ifstrempty{#3}{{#1}.{\sc{\numToLet{#2}}}}{{#1}.{\sc{\numToLet{#2}}}.{{\sc{\romannumeral #3}}}}}

\DeclareRobustCommand{\Case}[3]{\ifstrempty{#3}{Case {#1}.{\sc{\numToLet{#2}}}}{Case {#1}.{\sc{\numToLet{#2}}}.{{\sc{\romannumeral #3}}}}}

\DeclareRobustCommand{\Cases}[3]{\ifstrempty{#3}{Cases {#1}.{\sc{\numToLet{#2}}}}{Cases {#1}.{\sc{\numToLet{#2}}}.{{\sc{\romannumeral #3}}}}}

\title[Left-orderability and Kanenobu's knot]{On left-orderability and double branched covers of Kanenobu's knots}

\author[{F. Doria Medina
\and
M. Jackson
\and
J. Ruales
\and
H. Zeilberger}]{Fabian Doria Medina
\and
Michael Jackson
\and
Joaqu\'{i}n Ruales
\and
Hadas Zeilberger}

\begin{document}

\maketitle

\input{abstract}
\input{intro_text}
\input{body/body_text}

\input{bibliography_text}

\end{document}

%% file: abstract.tex
\begin{abstract}
We show that the fundamental group of the double branched cover of an infinite family of homologically thin, non-quasi-alternating knots is not left-orderable, giving further support for a conjecture of Boyer, Gordon, and Watson that an irreducible rational homology 3-sphere is an L-space if and only if its fundamental group is not left-orderable.
\end{abstract}

%% file: intro_text.tex
\section{Introduction}

Heegaard Floer homology is an invariant of 3-manifolds introduced by Ozsv\'{a}th and Szab\'{o} \cite{OzsvathSzabo}. In its simplest form, it associates to a closed 3-manifold $Y$ a graded $\mathbb{F}_2$ vector space, denoted $\widehat{HF}(Y)$. It was shown by Ozsv\'{a}th and Szab\'{o} \cite[Proposition 5.1]{OS2} that if $Y$ is a rational homology 3-sphere then
\begin{align*}
\textup{rk }\widehat{HF}(Y)\geq{}|H_1(Y ; \mathbb{Z})|.
\end{align*}

\begin{definition}An \emph{L-space} is a rational homology sphere $Y$ with simplest possible Heegaard Floer homology, that is, with
\begin{align*}
\textup{rk } \widehat{HF}(Y) = |H_1(Y ; \mathbb{Z})|.
\end{align*}
\end{definition}

Lens spaces are L-spaces, motivating the name. It is interesting to ask whether there exist alternative characterizations of L-spaces that do not depend on Heegaard Floer homology \cite[Question 11]{OS4}. We know that if $Y$ is an L-space then $Y$ does not admit a $C^{2}$ co-orientable, taut foliation \cite{OS3}. The non-existence of a co-orientable, taut foliation has been proposed by Ozsv\'{a}th and Szab\'{o} as a possible characterization of L-spaces. Along similar lines, a conjecture has been proposed \cite{BoyerGordonWatson} that attempts to characterize L-spaces through a property of their fundamental group.

\begin{conjecture_main}
An irreducible rational homology 3-sphere is an L-space if and only if its fundamental group is not left-orderable.
\end{conjecture_main}

Recall that a left-orderable group is a group which admits a left-invariant total order.

The conjectured relationship between L-spaces and left-orderability is already known for 3-manifolds that are double branched covers of non-split alternating links. It has been shown that for a non-split alternating link $K\subset S^3$, the fundamental group of $\Sigma{}(K)$ is not left-orderable \cite{BoyerGordonWatson} (cf. \cite{GreeneJE}, \cite{Ito}), where $\Sigma{}(K)$ denotes the double branched cover of $K\subset{}S^3$. Furthermore, Manolescu and Ozsv\'{a}th \cite{ManolescuOzsvath} showed that alternating links, and more generally, quasi-alternating knots, are homologically thin. In turn, Ozsv\'{a}th and Szab\'{o} showed in \cite{OzsvathandSzabo} that for a homologically thin link $K$, its double branched cover is an L-space. Therefore, if a 3-manifold $M$ is the double branched cover of some non-split alternating link, then $M$ is an L-space and $\pi_1(M)$ is not left-orderable.

We will verify that a specific class of L-spaces arising from the double branched covers of Kanenobu's knot (see Figure~\ref{figure:kanenobu}) have fundamental groups which are not left-orderable. We will consider the knots $K_n$ for $n\geq{}0$, defined as
\begin{align*}
K_n=K_{-10n,10n+3}.
\end{align*}
\noindent{}It was shown by Greene and Watson that $K_n$ is homologically thin (but not quasi-alternating), and so $\Sigma{}(K_n)$ is an L-space for $n\geq{}0$ \cite[Proposition 11]{GreeneWatson}. Generically, $K_n$ is non-alternating and $\Sigma{}(K_n)$ is hyperbolic and can not be obtained by surgery on a knot in $S^{3}$ \cite{HoffmanWalsh}, so these manifolds fall outside of the classes considered in \cite{BoyerGordonWatson}. The fundamental group, $G_n$, of the double branched cover of $K_n$ was computed by Greene and Watson \cite{GreeneWatson}, and has the following presentation:

\begin{align*}
G_n=\pi{}_1(\Sigma{}(K_n))=\langle a,\; b,\; c,\; d \mid{} &(a^{-1}b)^{10n}d^{-1}a^{2},\; b^{-2}c(b^{-1}a)^{10n},\\
&(d^{-1}c)^{10n+3}c^{-1}bc^{-2},\; d^{2}a^{-1}d(c^{-1}d)^{10n+3} \rangle
\end{align*}
where we have renamed the four generators $v_1$, $v_2$, $v_3$, and $v_4$ from the original paper as $a$, $b$, $c$, and $d$, respectively.

\begin{figure}[ht]
\includegraphics[scale=.37]{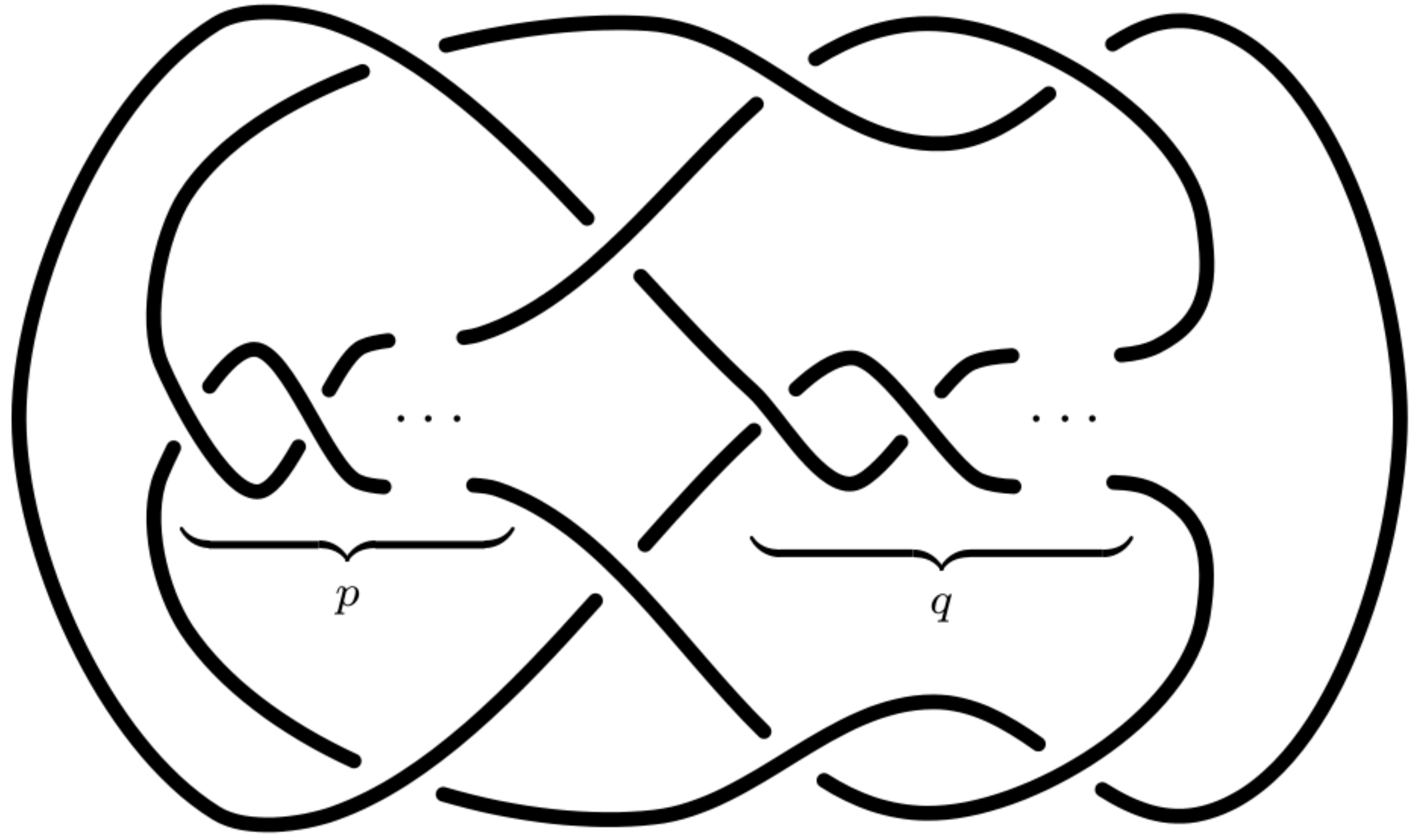}
\caption{Kanenobu's knot $K_{p,q}$. Image due to \cite{GreeneWatson}.}
\label{figure:kanenobu}
\end{figure}

\begin{theorem*} The fundamental group $G_n$ of the double branched cover of $K_n$ is not left-orderable.
\label{MAINTHEOREM}
\end{theorem*}

\noindent{} Next, we introduce various definitions and give background on left-orderability.

\subsection{Left-orderability}



\begin{definition}
A group $G$ is {\it left-orderable} if its elements can be given a left-invariant total order. That is, a total order $<$ such that $g<h$ implies $fg<fh$ for all $f, g, h\in{}G$.
\end{definition}

\begin{remark} By convention the trivial group is not left-orderable.
\end{remark}

We recall some facts on left-orderable groups from \cite{ClayRolfsen}.

\begin{fact} For some left-orderable group $(G, <)$ we can define {\it a} corresponding relation $>$ in the following way: for $g,h\in{}G$, $g>h$ if and only if $h<g$. This notational convenience will be used frequently.
\end{fact}

\begin{fact} In a left-orderable group G, $1<g$ (``$g$ is positive") if and only if $g^{-1}<1$ (``$g^{-1}$ is negative").
\label{fact:inverses}
\end{fact}

\begin{fact} Transitivity implies that in a left-orderable group products of positive elements are positive and products of negative elements are negative.
\label{1.3LO}
\end{fact}

\begin{proposition} In a left-orderable group $G$, $g\in{}G$ has the same sign as $g^{n}$ for any $n>1$.
\label{proposition:pospowers}
\end{proposition}
\begin{proof} Consequence of Fact \ref{1.3LO}.
\end{proof}

\begin{fact} A left-orderable group has no torsion.
\label{fact:torsion}
\end{fact}


\begin{fact} Let $G$ be a non-trivial group and let $g\in{}G$. There exists a left-ordering $<$ on $G$ such that $g<1$ if and only if there exists a left-ordering $<^\prime$ on $G$ such that $1<^\prime g$.
\label{fact:WLOG}
\end{fact}

\noindent{}We can also define left-orderability in a different way:

\begin{proposition} A group $G$ is left-orderable if and only if there exists a subset $P\subset{}G$ such that:
\begin{enumerate}
\item $P\cdot{}P\subset{}P$
\item $P\cap{}P^{-1}=\emptyset{}$
\item $G=P\cup{}P^{-1}\cup{}\{ 1 \}$
\end{enumerate}
\label{proposition:poscone}
\end{proposition}
\begin{proof} Suppose $G$ is left-orderable. Define
\begin{align*}
P=\Set*{g\in{}G}{1<g}.
\end{align*}
Then $P\cdot{}P\subset{}P$ since if $g>1$ and $h>1$ then $gh>g>1$. Therefore, $P$ satisfies the first condition. If $g>1$ then $g^{-1}<1$ and so $P\cap{}P^{-1}=\emptyset{}$. Therefore, $P$ satisfies the second condition. Finally, by the totality of a total ordering, all non-trivial elements in $G$ must be either positive or negative, thus $P\cup{}P^{-1}\cup{}\{ 1 \}$. Therefore, $P$ satisfies the third condition, completing one direction of the proof.\\
\\
Conversely, suppose there exists a subset $P\subset{}G$ satisfying the three conditions of the proposition. Define a left ordering in the following way:
\begin{align*}
g<h\Leftrightarrow{}g^{-1}h\in{}P.
\end{align*}
It is easy to check this defines a left-ordering.
\end{proof}

\begin{definition} For a group $G$, a subset $P\subset{}G$ satisfying the three conditions of Proposition~\ref{proposition:poscone} is called a {\it positive cone}.
\end{definition}

\subsection{Automated Proofs}
Several of the proofs for lemmas and propositions in this paper were generated by a computer program we created for the task. We will now briefly describe the algorithm our program employs, as it could be useful for future work in disproving the left-orderability of certain groups. Our program is similar to the program described in \cite[Section 8]{CalegariDunfield}.

For the proof that $G_n$ is not left-orderable, we argue by contradiction. That is, we assume that $G_n$ is left-orderable, thus for any left-ordering on $G_n$, there must exist a positive cone $P\subset{}G_n$. Based on Fact~\ref{fact:WLOG}, we can proceed under the assumption that $b^{-1}a\in{}P$, and then find additional elements of $G_n$ that must be contained in such a positive cone. With the addition of enough elements, we can in many cases reach a contradiction.

In order to accomplish this, the program takes two inputs:
\begin{enumerate}
\item A set $Q\subset{}P$ of elements that have been proven (either in previous iterations of the program, or by hand) to be contained in $P$, including $b^{-1}a$. This set will grow during the execution of the program, but we will ensure that it always is a subset of $P$, so $Q$ has the property inherited from $P$ that $1\not\in Q^{*}$, where $Q^*$ is the semigroup generated by $Q$. 
\item A subset $I$ of all words that we know are equal to the identity based on the group relations of $G_n$, closed under inversion and cyclic permutation. See (\ref{intro:cyclic}) for an example of what is meant by cyclic permutation.
\end{enumerate}
\begin{remark}
Words that are equal to the identity are henceforth referred to as identities.
\end{remark}
The four group relations of $G_n$ are obvious examples of identities. To give another example, Lemma~\ref{lemma:eq5} shows that $d^{-1}a^{2}b^{-2}c$ is also an identity. The cyclic permutations of this identity would be:
\begin{align}
\{d^{-1}a^{2}b^{-2}c,\; cd^{-1}a^{2}b^{-2},\; b^{-1}cd^{-1}a^{2}b^{-1},\nonumber{}\\
b^{-2}cd^{-1}a^{2},\; ab^{-2}cd^{-1}a,\; a^{2}b^{-2}cd^{-1}\}. \label{intro:cyclic}
\end{align}

Pseudocode for a simplified version of the program follows, where $A^*$ denotes the semigroup generated by the elements of $A$.\\

\begin{algorithmic}
\LOOP
\STATE $x\gets$ next nontrivial element of unknown sign
  \IF {$I\cap{}(Q\cup{}\{x\})^*=\emptyset{}$ \AND $I\cap{}(Q\cup{}\{x^{-1}\})^*\neq{\emptyset{}}$}
    \STATE {\textbf{add} $x$ \textbf{to} $Q$}
    \PRINT {$x$ added to positive list}
  \ELSIF {$I\cap{}(Q\cup{}\{x\})^*\neq{}\emptyset{}$ \AND $I\cap{}(Q\cup{}\{x^{-1}\})^*={\emptyset{}}$}
    \STATE {\textbf{add} $x^{-1}$ \textbf{to} $Q$}
    \PRINT {$x^{-1}$ added to positive list}
  \ELSIF {$I\cap{}(Q\cup{}\{x\})^*\neq{}\emptyset{}$ \AND $I\cap{}(Q\cup{}\{x^{-1}\})^*\neq{}{\emptyset{}}$}
    \PRINT {$x$ causes a contradiction}
    \STATE {\textbf{program halts}}
  \ENDIF
\ENDLOOP
\end{algorithmic}
$\;$\\
The ``next nontrivial element of unknown sign" from line 2 can either be user-input or computer-generated. Since there are infinite elements of unknown sign, we (or the computer) give preference to those elements with lowest word length, e.g. $c^{-1}d$ before $c^{-1}d^2$.

Within the program's \textbf{if} statements, we compute the intersection between the finite set $I$ and infinite semigroups generated by $Q\cup{}\{x\}$ or $Q\cup{}\{x^{-1}\}$. This is possible in finitely many operations because $I$ is finite and the semigroup is finitely generated. We use a method similar to using a deterministic finite automaton with the finitely generated semigroup as a language to check elements in $I$.

\subsection{Outline}
The paper is organized as follows. In Section~\ref{section:G_0}, we provide a proof that $G_0$ is not left-orderable. The case when $n=0$ is addressed separately because the proof for $n>0$ does not hold when $n=0$. The remainder of the paper is then devoted to a proof for the cases $n>0$. To facilitate the proof, we consider sixteen cases (see Table~\ref{table:casesAll}) based on the signs of the four generators of $G_n$ and disprove left-orderability in each. In Section~\ref{section:identityProofs} we show that the four generators of $G_n$ are non-trivial and distinct, justifying the totality of the sixteen cases we will address. In Section~\ref{section:generalLemmas} we prove lemmas that hold in all cases and that will be useful for later proofs. With these tools, left-orderability is straightforward to disprove in eleven of the sixteen cases, and we address these in Section~\ref{section:manyCases}. We disprove left-orderability in Cases 3, 4, 8, 1, and 16  in Sections~\ref{section:case3},~\ref{section:case4},~\ref{section:case8},~\ref{section:case1}, and~\ref{section:case16} respectively.

\subsection{Acknowledgements} We would like to thank Jennifer Hom and Kristen Hendricks for their generous advice throughout the project, Columbia University's REU Summer Program for providing us the opportunity to work together, Adam Clay and Dale Rolfsen for sharing their notes on ``Ordered Groups and Topology," and Tye Lidman and Liam Watson for suggesting this problem. We would also like to thank Tye Lidman and the anonymous reviewers at the Journal of Knot Theory and Its Ramifications for comments on earlier drafts of this paper. Finally, we would like to thank the National Science Foundation---Fabian Doria Medina was partly supported by NSF grant DMS-1149800.

%% file: body/body_text.tex
\input{body/G_0.tex}

\section{Non-triviality and distinctness of the four generators}
\label{section:identityProofs}

\noindent{}Recall that
\begin{align*}
G_n=\langle a,\; b,\; c,\; d \mid{} &(a^{-1}b)^{10n}d^{-1}a^{2},\; b^{-2}c(b^{-1}a)^{10n},\\
&(d^{-1}c)^{10n+3}c^{-1}bc^{-2},\; d^{2}a^{-1}d(c^{-1}d)^{10n+3} \rangle
\end{align*}
for integers $n\geq{}0$.

\vspace{10pt}

\begin{proposition}
The four generators of $G_n$ are all distinct and non-trivial.
\label{proposition:non-trivial}
\end{proposition}
\begin{proof}
Greene and Watson \cite[Section 4.2]{GreeneWatson} show that the abelianization of $G_n$ provides a homomorphism:
\begin{align*}
\phi:G_n\rightarrow \mathbb{Z}_{25}
\end{align*}
with $\phi(a)=13$, $\phi(b)=3$, $\phi(c)=6$, and $\phi(d)=1$. Because each of the generators maps to a distinct, non-identity element of $\mathbb{Z}_{25}$, it follows that the four generators are distinct and nontrivial.
\end{proof}

\begin{corollary}
Both $bc^{-3}$ and $ad^{-3}$ are non-trivial.
\label{corollary:non-trivial2}
\end{corollary}
\begin{proof}
By the third group relation, we have:
\begin{align*}
(d^{-1}c)^{10n+3}c^{-1}bc^{-2}&=1\\
\Rightarrow{}(cd^{-1})^{10n+3}bc^{-3}&=1\\
\Rightarrow{}bc^{-3}&=(dc^{-1})^{10n+3}.
\end{align*}
By Proposition~\ref{proposition:non-trivial}, $dc^{-1}\ne1$, therefore $bc^{-3}\ne1$.\\

\noindent{}By the fourth group relation, we have:
\begin{align*}
(d^{-1}c)^{10n+3}d^{-1}ad^{-2}&=1\\
\Rightarrow{}(cd^{-1})^{10n+3}ad^{-3}&=1\\
\Rightarrow{}ad^{-3}&=(dc^{-1})^{10n+3}.
\end{align*}
By Proposition~\ref{proposition:non-trivial}, $dc^{-1}\ne1$, therefore $ad^{-3}\ne1$.
\end{proof}

\noindent{}We will assume $G_n$ is left-orderable and then reach a contradiction. By Fact~\ref{fact:WLOG}, we can assume without loss of generality that $b^{-1}a\geq1$. Because of Proposition~\ref{proposition:non-trivial}, we can assume (without loss of generality) that $b^{-1}a>1$.\\

\noindent{}With four (non-trivial) group generators there are 16 possible cases for the signs of each of the generators that must be considered (see Table~\ref{table:casesAll}). We will disprove each of these cases.

\begin{table}[ht]
\begin{center}
\begin{tabular}{l | l | l | l | l}
Case\hspace{10 pt} & $a$\hspace{10 pt} & $b$\hspace{10 pt} & $c$\hspace{10 pt} & $d$\hspace{10 pt} \\\hline\hline
1 & $+$ & $+$ & $+$ & $+$ \\\hline
2 & $+$ & $+$ & $+$ & $-$ \\\hline
3 & $+$ & $+$ & $-$ & $+$ \\\hline
4 & $+$ & $+$ & $-$ & $-$ \\\hline
5 & $+$ & $-$ & $+$ & $+$ \\\hline
6 & $+$ & $-$ & $+$ & $-$ \\\hline
7 & $+$ & $-$ & $-$ & $+$ \\\hline
8 & $+$ & $-$ & $-$ & $-$ \\\hline
9 & $-$ & $+$ & $+$ & $+$ \\\hline
10 & $-$ & $+$ & $+$ & $-$ \\\hline
11 & $-$ & $+$ & $-$ & $+$ \\\hline
12 & $-$ & $+$ & $-$ & $-$ \\\hline
13 & $-$ & $-$ & $+$ & $+$ \\\hline
14 & $-$ & $-$ & $+$ & $-$ \\\hline
15 & $-$ & $-$ & $-$ & $+$ \\\hline
16 & $-$ & $-$ & $-$ & $-$ 
\end{tabular}
\end{center}
\caption{The 16 possible cases with 4 generators}
\label{table:casesAll}
\end{table}

\section{General Lemmas}
\label{section:generalLemmas}

\noindent{}First we prove some lemmas that will be used later. These lemmas are true for all cases listed in Table~\ref{table:casesAll}.

\begin{lemma} In $G_n$, $d^{-1}a^{2}=c^{-1}b^{2}$.
\label{lemma:eq5}
\end{lemma}
\begin{proof} The first group relation for $G_n$ can be rearranged as follows:
\begin{align}
(a^{-1}b)^{10n}d^{-1}a^{2}&=1\nonumber{}\\
\Rightarrow{}(b^{-1}a)^{10n}&=d^{-1}a^{2},\label{eq1:2}
\end{align}
and the second group relation can be rearranged as follows:
\begin{align}
b^{-2}c(b^{-1}a)^{10n}&=1\nonumber{}\\
\Rightarrow{}(b^{-1}a)^{10n}&=c^{-1}b^{2}.\label{eq2:2}
\end{align}
This shows that:
\begin{align*}
d^{-1}a^{2}&=c^{-1}b^{2}.\qedhere
\end{align*}
\end{proof}

\begin{corollary} In $G_n$, $a^{2}b^{-2}=dc^{-1}$.
\label{corollary:eq6}
\end{corollary}
\begin{proof}By Lemma~\ref{lemma:eq5} we know:
\begin{align*}
d^{-1}a^{2}&=c^{-1}b^{2}\\
\Rightarrow{}a^{2}b^{-2}&=dc^{-1}.\qedhere
\end{align*}
\end{proof}

\begin{lemma} In $G_n$, $d^{2}a^{-1}d=c^{2}b^{-1}c$.
\label{lemma:eq7}
\end{lemma}
\begin{proof}
The third group relation for $G_n$ can be rearranged as follows:
\begin{align*}
(d^{-1}c)^{10n+3}c^{-1}bc^{-2}&=1\\
\Rightarrow{}(d^{-1}c)^{10n+3}&=c^{2}b^{-1}c,
\end{align*}
and the fourth group relation can be rearranged as follows:
\begin{align*}
d^{2}a^{-1}d(c^{-1}d)^{10n+3}&=1\\
\Rightarrow{}(d^{-1}c)^{10n+3}&=d^{2}a^{-1}d.
\end{align*}
Showing that:
\begin{align}
\label{eq7}d^{2}a^{-1}d&=c^{2}b^{-1}c.
\end{align}
\end{proof}

\begin{lemma} In $G_n$, $d^{2}a=c^{2}b$.
\label{lemma:eq8}
\end{lemma}
\begin{proof}
Starting from Lemma~\ref{lemma:eq5}, we have:
\begin{align}
d^{-1}a^{2}&=c^{-1}b^{2}\nonumber{}\\
\Rightarrow{}c&=b^{2}a^{-2}d.\label{eq5:2}
\end{align}
Using (\ref{eq5:2}), we can substitute for the right-most $c$ in (\ref{eq7}). We find:
\begin{align*}
d^{2}a^{-1}d&=c^{2}b^{-1}b^{2}a^{-2}d\\
\Rightarrow{}d^{2}a^{-1}d&=c^{2}ba^{-2}d\\
\Rightarrow{}d^{2}a&=c^{2}b.\qedhere
\end{align*}
\end{proof}

\begin{lemma} In $G_n$, $ad^{-1}a=bc^{-1}b$.
\label{lemma:eq16}
\end{lemma}
\begin{proof} By the first group relation, we have:
\begin{align}
(a^{-1}b)^{10n}d^{-1}a^{2}&=1\nonumber{}\\
\Rightarrow{}(a^{-1})(ba^{-1})^{10n}ad^{-1}a^{2}&=1\nonumber{}\\
\Rightarrow{}(ba^{-1})^{10n}ad^{-1}a&=1.\label{eq1:3}
\end{align}
By the second group relation, we have:
\begin{align}
b^{-2}c(b^{-1}a)^{10n}&=1\nonumber{}\\
\Rightarrow{} b^{-2}cb^{-1}(ab^{-1})^{10n}b&=1\nonumber{}\\
\Rightarrow{}b^{-1}cb^{-1}(ab^{-1})^{10n}&=1\nonumber{}\\
\Rightarrow{}b^{-1}cb^{-1}&=(ba^{-1})^{10n}.\label{eq2:3}
\end{align}
Substituting (\ref{eq2:3}) into (\ref{eq1:3}), we see:
\begin{align*}
b^{-1}cb^{-1}ad^{-1}a&=1\\
\Rightarrow{}ad^{-1}a&=bc^{-1}b.\qedhere
\end{align*}
\end{proof}

\begin{lemma} In $G_n$, $c^{-3}d^{3}=b^{-1}a$.
\label{lemma:eq9}
\end{lemma}
\begin{proof} By the third group relation, we have:
\begin{align}
(d^{-1}c)^{10n+3}c^{-1}bc^{-2}&=1\nonumber{}\\
\Rightarrow{}c^{-1}(cd^{-1})^{10n+3}bc^{-2}&=1\nonumber{}\\
\Rightarrow{}(cd^{-1})^{10n+3}bc^{-3}&=1.\label{eq3:33}
\end{align}
By the fourth group relation, we have:
\begin{align}
d^{2}a^{-1}d(c^{-1}d)^{10n+3}&=1\nonumber{}\\
\Rightarrow{}d^{2}a^{-1}(dc^{-1})^{10n+3}d&=1\nonumber{}\\
\Rightarrow{}d^{3}a^{-1}(dc^{-1})^{10n+3}&=1\nonumber{}\\
\Rightarrow{}d^{3}a^{-1}&=(cd^{-1})^{10n+3}.\label{eq4:33}
\end{align}
Substituting (\ref{eq4:33}) into (\ref{eq3:33}), we see:
\begin{align*}
d^{3}a^{-1}bc^{-3}&=1\\
\Rightarrow{}c^{-3}d^{3}&=b^{-1}a.\qedhere
\end{align*}
\end{proof}

\begin{corollary} If $G_n$ is left-orderable, then $c^{-3}d^{3}>1$.
\label{corollary:inEq5}
\end{corollary}
\begin{proof}This is an immediate consequence of Lemma~\ref{lemma:eq9} and the general assumption that $b^{-1}a>1$.
\end{proof}

\begin{lemma} If $G_n$ is left-orderable, then $d^{-1}a^{2}>1$.
\label{lemma:inEq3}
\end{lemma}
\begin{proof}By (\ref{eq1:2}) we have:
\begin{align*}
(b^{-1}a)^{10n}&=d^{-1}a^{2},
\end{align*}
and since we are assuming $b^{-1}a>1$, this shows that $d^{-1}a^{2}>1$.
\end{proof}

\begin{lemma} If $G_n$ is left-orderable, then $d^{-2}c^{2}b^{-1}cd^{-1}b<1$.
\label{lemma:inEq7}
\end{lemma}
\begin{proof} By  Lemma~\ref{lemma:eq7} we have:
\begin{align}
d^{2}a^{-1}d&=c^{2}b^{-1}c\nonumber{}\\
\Rightarrow{}d^{2}&=c^{2}b^{-1}cd^{-1}a.\label{eq7:2}
\end{align}
By rearranging the result of Lemma~\ref{lemma:eq9} we have:
\begin{align*}
c^{3}=d^{3}a^{-1}b=d(c^{2}b^{-1}cd^{-1}a)a^{-1}b.
\end{align*}
Note that the last equality follows by substituting for $d^{2}$ using (\ref{eq7:2}). This expression can be rearranged to give:
\begin{align}
d&=c^{3}b^{-1}dc^{-1}bc^{-2}\nonumber{}\\
\Rightarrow{}c^{-1}d&=c^2b^{-1}dc^{-1}bc^{-2}.\label{eq13}
\end{align}
Now (\ref{eq13}) can be rearranged to yield:
\begin{align}
d^{-1}c^3b^{-1}dc^{-1}bc^{-2}&=1\nonumber{}\\
\Rightarrow{}(d^{-3}c^3)(b^{-1}dc^{-1}bc^{-2}d^{2})&=1.\label{lemma:inEq7:contradiction}
\end{align}
By Corollary~\ref{corollary:inEq5}, $d^{-3}c^3<1$, therefore (\ref{lemma:inEq7:contradiction}) shows that:
\begin{align*}
b^{-1}dc^{-1}bc^{-2}d^{2}&>1\\
\Rightarrow{}d^{-2}c^{2}b^{-1}cd^{-1}b&<1.\qedhere
\end{align*}
\end{proof}

\begin{lemma} If $G_n$ is left-orderable, then $c^{-1}ba^{-1}d>1$.
\label{lemma:inEq:CbAd}
\end{lemma}
\begin{proof} By Lemma~\ref{lemma:eq16}, we have:
\begin{align}
a^{-1}da^{-1}bc^{-1}b&= 1\nonumber{}\\
\Rightarrow{}c^{-1}ba^{-1}da^{-1}b &=1\nonumber{}\\
\Rightarrow{}(b^{-1}a)(d^{-1}ab^{-1}c) &= 1.\label{lemma:inEq:CbAd:contradiction}
\end{align}
We know $b^{-1}a>1$, so (\ref{lemma:inEq:CbAd:contradiction}) shows that:
\begin{align*}
d^{-1}ab^{-1}c&<1\\
\Rightarrow{}c^{-1}ba^{-1}d&>1.\qedhere
\end{align*}
\end{proof}

\begin{lemma} If $G_n$ is left-orderable, then $b<1$ implies $c<1$.
\label{genImp1}
\end{lemma}
\begin{proof}By Lemma~\ref{lemma:inEq3}, $d^{-1}a^{2}>1$, and hence $a^{-2}d<1$. By Lemma~\ref{lemma:eq5}, we have:
\begin{align*}
d^{-1}a^{2}&=c^{-1}b^{2}\\
\Rightarrow{}c&=b^{2}a^{-2}d,
\end{align*}
and it is now easy to see that:
\begin{align*}
b<1\Rightarrow{}b^{2}<1\Rightarrow{}b^{2}a^{-2}d&<1\Rightarrow{}c<1.\qedhere
\end{align*}
\end{proof}

\input{body/cases2-15}


\input{body/case_1}

\input{body/case_16}

\noindent{}With Proposition~\ref{proposition:case16.B}, we have eliminated the final sub-case of Case 16. We have therefore shown that $G_n$ is not left-orderable.

%% file: body/G_0.tex
\section{Proof that $G_0$ is not left-orderable}
\label{section:G_0}
We start by proving that $G_0$ is not left-orderable, as the proof uses a different approach than the general case $n\geq1$
\begin{lemma} $G_0$ is isomorphic to $\langle x,\; y\mid{}(x^{-2}y^{2})^{3}=x^{5}=y^{5} \rangle$.
\end{lemma}
\begin{proof} When $n=0$, we have:
\begin{align*}
G_0=\langle a,\; b,\; c,\; d \mid{} d^{-1}a^{2},\; b^{-2}c,\; (d^{-1}c)^{3}c^{-1}bc^{-2},\; d^{2}a^{-1}d(c^{-1}d)^{3} \rangle.
\end{align*}
The first two relations show that $a^{2}=d$ and $b^{2}=c$, thus we can rewrite the presentation using only $a$ and $b$ as generators:
\begin{align*}
G_0&=\langle a,\;b\mid{} (a^{-2}b^{2})^{3}b^{-2}bb^{-4},\; a^{4}a^{-1}a^{2}(b^{-2}a^{2})^{3} \rangle\\
G_0&=\langle a,\;b\mid{} (b^{-2}a^{2})^{3}=b^{-5},\; a^{5}(b^{-2}a^{2})^{3} \rangle\\
G_0&=\langle a,\;b\mid{} (b^{-2}a^{2})^{3}=b^{-5},\; a^{5}b^{-5} \rangle\\
G_0&=\langle a,\;b\mid{} (a^{-2}b^{2})^{3}=a^{5}=b^{5} \rangle.\qedhere
\end{align*}
\end{proof}

\begin{lemma} Both $x^{5}$ and $y^{5}$ commute with all elements in $G_0$.
\label{lemma:G_0:center}
\end{lemma}
\begin{proof} Since we can change $x^5$ to $y^5$ and back as necessary, it is clear that both $x^5$ and $y^5$ commute with $x$, $y$, $x^{-1}$, and $y^{-1}$ and therefore with any element of $G_0$.

\end{proof}

\begin{lemma} If $G_0$ is left-orderable, then $wx^{n}w^{-1}$ has the same sign as $x$ for any $w\in{}G_0$ and for any $n\geq{}1$. Similarly, $wy^{n}w^{-1}$ has the same sign as $y$ for any $w\in{}G_0$ and for any $n\geq{}1$.
\label{lemma:G_0:conjugates}
\end{lemma}
\begin{proof} Suppose $G_0$ is left-orderable. We know by Lemma~\ref{lemma:G_0:center} that for any $w\in{}G_0$:
\begin{align*}
wx^{5}w^{-1}=ww^{-1}x^{5}=x^{5}.
\end{align*}
By Proposition~\ref{proposition:pospowers}, $x^5$ has the same sign as $x$, and thus $wx^{5}w^{-1}$ has the same sign as $x$. But $wx^{5}w^{-1}=(wxw^{-1})^{5}$, and so by Proposition~\ref{proposition:pospowers} $wxw^{-1}$ has the same sign as $wx^{5}w^{-1}$ and therefore has the same sign as $x$. A similar proof works for $y$.
\end{proof}

\begin{lemma} If $G_0$ is left-orderable and $x>1$, then $x^{-2}y^{2}>1$.
\label{lemma:G_0:XXyy}
\end{lemma}
\begin{proof}
Suppose $G_0$ is left-orderable and suppose $x>1$. Then $x^5>1$ and $(x^{-2}y^2)^{3}>1$ since $(x^{-2}y^2)^{3}=x^5$. By Proposition~\ref{proposition:pospowers} this shows that $x^{-2}y^2>1$.
\end{proof}

\begin{proposition}
The group $G_0=\langle x,\; y \mid{} (x^{-2}y^{2})^{3}=x^{5}=y^{5} \rangle$ is not left-orderable.
\label{propG0}
\end{proposition}
\begin{proof} Suppose (for contradiction) that $G_0$ is left-orderable. First note that if $x=1$, then $y^5=1$. Then either $y=1$ as well and $G_0$ is trivial, or $y\neq{}1$ and $G_0$ has torsion and is therefore not left-orderable by Fact~\ref{fact:torsion}, a contradiction. Thus by Fact~\ref{fact:WLOG} we can assume without loss of generality that $x>1$. Note that $x>1$ implies $x^5=y^5>1$ which implies $y>1$ by Proposition~\ref{proposition:pospowers}. Starting with the group relation, we have:
\begin{align}
x^{-2}y^{2}x^{-2}y^{2}x^{-2}y^{2}&=y^{5}\nonumber{}\\
x^{-2}y^{2}x^{-2}y^{2}x^{-2}&=y^{3}\nonumber{}\\
x^{-2}y^{2}x^{-2}y^{2}x^{2}&=y^{3}x^{4}\nonumber{}\\
x^{-2}y^{2}x^{-2}y^{2}x^{2}&=x^{5}y^{3}x^{-1}
\label{x5center}\\
y^{2}x^{-2}y^{2}x^{2} & =x^{7}y^{3}x^{-1}\nonumber{}\\
y^{-2}x^{-2}y^{2}x^{2} & =y^{-4}x^{7}y^{3}x^{-1}\nonumber{}\\
[x^{-2}y^{-2}]x^{-2}[y^{2}x^{2}] & =x^{-2}y^{-4}x^{5}x^{2}y^{3}x^{-1}\nonumber{}\\
& =x^{-2}yx^{2}y^{3}x^{-1}
\label{x5=y5}\\
[(x^{-1})x^{-2}y^{-2}]x^{-2}[y^{2}x^{2}(x)] & =(x^{-1})x^{-2}yx^{2}y^{3}x^{-1}(x)\nonumber{}\\
[(y^{3})x^{-3}y^{-2}]x^{-2}[y^{2}x^{3}(y^{-3})] & =(y^{3})x^{-3}yx^{2}(y^{3}y^{-3})\nonumber{}\\
[(x^{-2})y^{3}x^{-3}y^{-2}]x^{-2}[y^{2}x^{3}y^{-3}(x^{2})] & =(x^{-2})y^{3}x^{-3}yx^{2}(x^{2})\nonumber{}\\
[x^{-2}y^{3}x^{-3}y^{-2}] x^{-2} [x^{-2}y^{3}x^{-3}y^{-2}]^{-1} & =[x^{-2}y^{2}]y[(x^{-3})y(x^{3})]x.
\label{PFcontraG0}
\end{align}
Where for (\ref{x5center}) we have used the fact (shown in Lemma~\ref{lemma:G_0:center}) that $x^5$ commutes with any element of $G_0$, for (\ref{x5=y5}) we have used the group relation $x^{5}=y^{5}$. Now in (\ref{PFcontraG0}), the right expression must be positive since $x>1$, $y>1$, $x^{-2}y^{2}>1$ by Lemma~\ref{lemma:G_0:XXyy}, and $(x^{-3})y(x^{3})>1$ by Lemma~\ref{lemma:G_0:conjugates}. However, the expression on the left is negative by Lemma~\ref{lemma:G_0:conjugates} since it is of the form $(wx^{-1}w^{-1})^{2}$ for some $w\in{}G_0$. This is a contradiction.
\end{proof}

\begin{remark}
An alternative proof of Proposition \ref{propG0} follows from the fact that $K_0$ is a Montesino knot, and hence $\Sigma{}(K_0)$ is a Seifert fibered space. Proposition \ref{propG0} then follows from \cite[Theorem 4]{BoyerGordonWatson}.
\end{remark}

%% file: body/cases2-15.tex
\section{Cases 2, 5, 6, 7, 9, 10, 11, 12, 13, 14, and 15}
\label{section:manyCases}

\begin{proposition} If $G_n$ is left-orderable, then Cases 9, 10, 11, and 12 are impossible.\label{proposition:case9,10,11,12}
\end{proposition}
\begin{proof} Suppose $G_n$ is left-orderable. In Cases 9, 10, 11, and 12, $a<1$ and $b>1$ and so $a<b$, but we have taken $b^{-1}a$ to be positive, telling us that $a>b$. Therefore, Cases 9, 10, 11, and 12 are not possible.
\end{proof}

\begin{proposition} If $G_n$ is left-orderable, then Cases 2 and 15 are impossible.\label{proposition:case2,15}
\end{proposition}
\begin{proof} Suppose $G_n$ is left-orderable. The fourth group relation for $G_n$ tells us that:
\begin{align}
d^{2}a^{-1}d(c^{-1}d)^{10n+3}&=1.\label{eq4}
\end{align}
Suppose (for contradiction) that the signs of the generators are as in Case 2. Then, $d<1$, while $a>1$ implies $a^{-1}<1$ and $c>1$ implies $c^{-1}<1$. These statements show that:
\begin{align*}
d^{2}a^{-1}d(c^{-1}d)^{10n+3}<1.
\end{align*}
This contradicts (\ref{eq4}); therefore, Case 2 is impossible.\\
\noindent{}Suppose now (for contradiction) that the signs of the generators are as in Case 15. Then, $d>1$, while $a<1$ implies $a^{-1}>1$, and $c<1$ implies $c^{-1}>1$. These statements show that:
\begin{align*}
d^{2}a^{-1}d(c^{-1}d)^{10n+3}>1.
\end{align*}
This contradicts (\ref{eq4}); therefore, Case 15 is impossible.
\end{proof}

\begin{proposition} If $G_n$ is left-orderable, then Case 7 is impossible.
\label{proposition:case7}
\end{proposition}
\begin{proof} Suppose that $G_n$ is left-orderable. By Lemma~\ref{lemma:eq8} we have:
\begin{align*}
d^{2}a&=c^{2}b.
\end{align*}
This shows that if $a$ and $d$ are both positive then $b$ and $c$ cannot both be negative, eliminating Case 7 as a possibility.
\end{proof}

\begin{proposition} If $G_n$ is left-orderable, then Cases 5, 6, 13, and 14 are impossible.\label{proposition:case5,6,13,14}
\end{proposition}
\begin{proof} Suppose that $G_n$ is left-orderable. By Lemma~\ref{genImp1}, Cases 5, 6, 13, and 14 are not possible, since in these cases $b<1$ but $c>1$.
\end{proof}
\vspace{20 pt}
\noindent{}To summarize, the remaining cases are 1, 3, 4, 8, and 16. They are shown in Table~\ref{table:casesPart1}.

\begin{table}[ht]
\begin{center}
\begin{tabular}{l | l | l | l | l}
Case\hspace{10 pt} & $a$\hspace{10 pt} & $b$\hspace{10 pt} & $c$\hspace{10 pt} & $d$\hspace{10 pt} \\\hline\hline
1 & $+$ & $+$ & $+$ & $+$ \\\hline
3 & $+$ & $+$ & $-$ & $+$ \\\hline
4 & $+$ & $+$ & $-$ & $-$ \\\hline
8 & $+$ & $-$ & $-$ & $-$ \\\hline
16 & $-$ & $-$ & $-$ & $-$
\end{tabular}
\end{center}
\caption{The five cases that remain after considering Propositions~\ref{proposition:case9,10,11,12},~\ref{proposition:case2,15},~\ref{proposition:case7}, and~\ref{proposition:case5,6,13,14}.}
\label{table:casesPart1}
\end{table}


\section{Case 3}
\label{section:case3}

\noindent{}We will now show that if $G_n$ is left-orderable then the four generators cannot have the signs shown in Case 3 (see Table~\ref{table:case3}). To accomplish this we will assume that $G_n$ is left-orderable and that the signs of the generators are as in Case 3 and reach a contradiction. 

\begin{table}[ht]
\begin{center}
\begin{tabular}{l | l | l | l | l}
Case\hspace{10 pt} & $a$\hspace{10 pt} & $b$\hspace{10 pt} & $c$\hspace{10 pt} & $d$\hspace{10 pt} \\\hline\hline
3 & $+$ & $+$ & $-$ & $+$ 
\end{tabular}
\end{center}
\caption{The signs of the four generators in Case 3}
\label{table:case3}
\end{table}

\begin{lemma} In Case 3, $ba^{-1}>1$.
\label{lemma:inEq3.1}
\end{lemma}
\begin{proof}By Lemma~\ref{lemma:eq8}, we have:
\begin{align}
c^{2}b&=d^{2}a\nonumber{}\\
\Rightarrow{}(c^{-2}d^{2})(ab^{-1})&=1.\label{eq8:2}
\end{align}
In Case 3, we assume that $c<1$ and $d>1$, thus $c^{-2}d^{2}>1$, and thus by (\ref{eq8:2}) we see that:
\begin{align*}
ab^{-1}&<1\\
\Rightarrow{}ba^{-1}&>1.\qedhere
\end{align*}
\end{proof}

\begin{proposition}
If $G_n$ is left-orderable, then Case 3 ($a>1$, $b>1$, $c<1$, and $d>1$) is impossible.
\end{proposition}
\begin{proof}Suppose that $G_n$ is left-orderable, and suppose (for contradiction) that $a>1$, $b>1$, $c<1$, and $d>1$. By the second group relation, we have:
\begin{align}
b^{-2}c(b^{-1}a)^{10n}&=1\nonumber{}\\
\Rightarrow{}b(a^{-1}b)^{10n}c^{-1}b&=1\nonumber{}\\
\Rightarrow{}(ba^{-1})^{10n}(bc^{-1}b)&=1.\label{proposition:case3:contradiction}
\end{align}
Now by Lemma~\ref{lemma:inEq3.1}, $(ba^{-1})^{10n}>1$. Further, we are assuming that $b>1$ and $c<1$ so $bc^{-1}b>1$. Therefore, (\ref{proposition:case3:contradiction}) states that a product of positive elements equals the identity, a contradiction.
\end{proof}


\section{Case 4}
\label{section:case4}

\noindent{}We will now show that if $G_n$ is left-orderable then the four generators cannot have the signs shown in Table~\ref{table:case4}. To accomplish this we will assume that $G_n$ is left-orderable and that the signs of the generators are as in Case 4 and reach a contradiction.

\begin{table}[ht]
\begin{center}
\begin{tabular}{l | l | l | l | l}
Case\hspace{10 pt} & $a$\hspace{10 pt} & $b$\hspace{10 pt} & $c$\hspace{10 pt} & $d$\hspace{10 pt} \\\hline\hline
4 & $+$ & $+$ & $-$ & $-$ 
\end{tabular}
\end{center}
\caption{The signs of the four generators in Case 4.}
\label{table:case4}
\end{table}

\begin{lemma} In Case 4, $c^{-1}d>1$.
\label{inEq4.1}
\end{lemma}
\begin{proof} By the third group relation, we have:
\begin{align}
(d^{-1}c)^{10n+3}(c^{-1}bc^{-2})&=1.\label{lemma:inEq4.1:contradiction}
\end{align}
In Case 4, $b>1$, and $c^{-1}>1$, thus (\ref{lemma:inEq4.1:contradiction}) shows that $d^{-1}c<1$ or equivalently $c^{-1}d>1$.
\end{proof}

\begin{lemma} In Case 4, $ab^{-1}>1$.
\label{inEq4.2}
\end{lemma}
\begin{proof} By the first group relation, we have:
\begin{align}
(a^{-1}b)^{10n}d^{-1}a^2&=1\nonumber{}\\
\Rightarrow a^{-1}(ba^{-1})^{10n}ad^{-1}a^{2}&=1\nonumber{}\\
\Rightarrow(ba^{-1})^{10n}(a)(d^{-1})(a)&=1.\label{eq1:4}
\end{align}
In Case 4, $a>1$ and $d^{-1}>1$, thus (\ref{eq1:4}) shows that $ba^{-1}<1$, or equivalently $ab^{-1}>1$.
\end{proof}

\begin{lemma} In Case 4, $c^{-2}d^{2}>1$.\label{inEq4.3}
\end{lemma}
\begin{proof} By the third group relation, we have:
\begin{align}
(d^{-1}c)^{10n+3}c^{-1}bc^{-2}&=1\nonumber{}\\
\Rightarrow(bc^{-2})(d^{-1}c)^{10n+3}(c^{-1})&=1\nonumber{}\\
\Rightarrow(bc^{-2})(cc^{-1})((dd^{-1})d^{-1}c)^{10n+3}(cc^{-1})(c^{-1})&=1\nonumber{}\\
\Rightarrow(bc^{-2})c(c^{-1}dd^{-2}cc)^{10n+3}c^{-1}(c^{-1})&=1\nonumber{}\\
\Rightarrow(b)(c^{-1})([c^{-1}d][d^{-2}c^2])^{10n+3}c^{-2}&=1.\label{lemma:inEq4.3:contradiction}
\end{align}
In Case 4, $b>1$, $c^{-1}>1$, and $c^{-1}d>1$ by Lemma~\ref{inEq4.1}. Therefore, (\ref{lemma:inEq4.3:contradiction}) shows that $d^{-2}c^2<1$, or equivalently, $c^{-2}d^2>1$.
\end{proof}

\begin{proposition} If $G_n$ is left-orderable, then Case 4 ($a>1$, $b>1$, $c<1$, and $d<1$) is impossible.
\end{proposition}
\begin{proof} Suppose $G_n$ is left-orderable. By Lemma~\ref{lemma:eq8}, we have:
\begin{align}
d^2a&=c^2b\nonumber{}\\
\Rightarrow b^{-1}c^{-2}d^2a&=1\nonumber{}\\
\Rightarrow (c^{-2}d^2)(ab^{-1})&=1.\label{proposition:case4:contradiction}
\end{align}
But $ab^{-1}>1$ by Lemma~\ref{inEq4.2} and $c^{-2}d^2>1$ by Lemma~\ref{inEq4.3}, thus (\ref{proposition:case4:contradiction}) is a contradiction.
\end{proof}


\section{Case 8}
\label{section:case8}

\noindent{}We will now show that if $G_n$ is left-orderable then the four generators cannot have the signs shown in Table~\ref{table:case8}. To accomplish this we will assume that $G_n$ is left-orderable and that the signs of the generators are as in Case 8 and reach a contradiction. 

\begin{table}[ht]
\begin{center}
\begin{tabular}{l | l | l | l | l}
Case\hspace{10 pt} & $a$\hspace{10 pt} & $b$\hspace{10 pt} & $c$\hspace{10 pt} & $d$\hspace{10 pt} \\\hline\hline
8 & $+$ & $-$ & $-$ & $-$ \\
\end{tabular}
\end{center}
\caption{The signs of the four generators in Case 8.}
\label{table:case8}
\end{table}

\begin{lemma} In Case 8, $bc^{-1}>1$.
\label{lemma:inEq8.2}
\end{lemma}
\begin{proof} In Case 8, we have:
\begin{align}
ab^{-1}>1,\label{inEq8.1}
\end{align} 
since $a>1$ and $b^{-1}>1$. The second group relation of $G_n$ tells us that:
\begin{align}
b^{-2}c(b^{-1}a)^{10n}&=1\nonumber{}\\
\Rightarrow{}b^{-1}c(b^{-1}a)^{10n}b^{-1}&=1\nonumber{}\\
\Rightarrow{}(b^{-1})(cb^{-1})(ab^{-1})^{10n}&=1.\label{lemma:inEq8.2:contradiction}
\end{align}
In Case 8, $b^{-1}>1$ and $ab^{-1}>1$ by (\ref{inEq8.1}), therefore (\ref{lemma:inEq8.2:contradiction}) tells us that:
\begin{align*}
cb^{-1}&<1\\
\Rightarrow{}bc^{-1}&>1.\qedhere
\end{align*}
\end{proof}

\begin{lemma} In Case 8, $c^{-1}b>1$.
\label{lemma:inEq8.3}
\end{lemma}
\begin{proof} The second group relation of $G_n$ tells us that:
\begin{align}
b^{-2}c(b^{-1}a)^{10n}&=1\nonumber{}\\
\Rightarrow{}(b^{-1}c)(b^{-1}a)^{10n}b^{-1}&=1.\label{lemma:inEq8.3:contradiction}
\end{align}
In Case 8, $b<1\Rightarrow{}b^{-1}>1$ and we have assumed in general that $b^{-1}a>1$, thus (\ref{lemma:inEq8.3:contradiction}) tells us that:
\begin{align*}
b^{-1}c&<1\\
\Rightarrow{}c^{-1}b&>1.\qedhere
\end{align*}
\end{proof}

\begin{lemma} In Case 8, $c^{-1}d>1$.
\label{lemma:inEq8.4}
\end{lemma}
\begin{proof} In Case 8, $c<1$ so $c^{-2}>1$, and Lemma~\ref{lemma:inEq8.3} states $c^{-1}b>1$. Therefore, we have:
\begin{align}
c^{-1}bc^{-2}>1.\label{inEq8.5}
\end{align}
The third group relation tells us that:
\begin{align}
(d^{-1}c)^{10n+3}c^{-1}bc^{-2}&=1\nonumber{}\\
\Rightarrow{}(d^{-1}c)^{10n+3}(c^{-1}bc^{-2})&=1.\label{lemma:inEq8.4:contradiction}
\end{align}
Together, (\ref{inEq8.5}) and (\ref{lemma:inEq8.4:contradiction}) show that:
\begin{align*}
d^{-1}c&<1\\
\Rightarrow{}c^{-1}d&>1.\qedhere
\end{align*}
\end{proof}

\begin{lemma} In Case 8, $d^{-2}c^{2}>1$.
\label{lemma:inEq8.6} 
\end{lemma}
\begin{proof} By Lemma~\ref{lemma:eq8}, we have:
\begin{align}
d^{2}a&=c^{2}b\nonumber{}\\
\Rightarrow{}(ab^{-1})(c^{-2}d^{2})&=1.\label{lemma:inEq8.6:contradiction}
\end{align}
In Case 8, $b<1$ and $a>1$ so $ab^{-1}>1$. Therefore, (\ref{lemma:inEq8.6:contradiction}) shows that:
\begin{align*}
c^{-2}d^{2}&<1\\
\Rightarrow{}d^{-2}c^{2}&>1.\qedhere
\end{align*}
\end{proof}

\begin{proposition} If $G_n$ is left-orderable, then Case 8 ($a>1$, $b<1$, $c<1$, and $d<1$) is impossible.
\label{proposition:case8}
\end{proposition}
\begin{proof} Suppose that $G_n$ is left-orderable, and suppose (for contradiction), that $a>1$, $b<1$, $c<1$, and $d<1$.
By the third group relation, we have: 
\begin{align}
(d^{-1}c)^{10n+3}c^{-1}bc^{-2}&=1\nonumber{}\\
\Rightarrow{}(bc^{-2})((d^{-1}c)^{10n+3})(c^{-1})&=1\nonumber{}\\
\Rightarrow{}(bc^{-2})(cc^{-1})(((dd^{-1})d^{-1}c)^{10n+3})(cc^{-1})(c^{-1})&=1\nonumber{}\\
\Rightarrow{}(bc^{-2}c)(c^{-1}dd^{-2}cc)^{10n+3}(c^{-1})(c^{-1})&=1\nonumber{}\\
\Rightarrow{}(bc^{-1})([c^{-1}d][d^{-2}c^{2}])^{10n+3}(c^{-2})&=1.\label{proposition:case8:contradiction}
\end{align}
By Lemma~\ref{lemma:inEq8.2}, $bc^{-1}>1$, by Lemma~\ref{lemma:inEq8.4}, $c^{-1}d>1$, and by Lemma~\ref{lemma:inEq8.6}, $d^{-2}c^{2}>1$. Furthermore, in Case 8 $c<1$ so $c^{-2}>1$. Therefore, (\ref{proposition:case8:contradiction}) states that a product of positives is the identity, a contradiction.
\end{proof}

%% file: body/case_1.tex
\section{Case 1}
\label{section:case1}

\noindent{}It would be possible to disprove left-orderability in Cases 1 and 16 by finding a word $w$ that admits only positive occurrences of the generators $a$, $b$, $c$, and $d$ and that satisfies $w=1$ when considered as an element of $G_n$. We were not able to find such a word or disprove its existence. However, certain qualities of the group relations of $G_n$ suggest that no such word $w$ exists.\\

\noindent{}First, each group relation contains a pair of generators, one raised to a positive power and one raised to a negative power, that are repeated a number of times depending on $n$. For example, $(a^{-1}b)^{10n}$ as part of $(a^{-1}b)^{10n}d^{-1}a^{2}$.\\

\noindent{}Second, while these repeated pairs are easily eliminated through the combination of two relations, the resultant identities do not improve the situation. By combining relation 1 with relation 2 and relation 3 with relation 4, we get $a^{2}b^{-2}cd^{-1}$ and $dc^{-1}bc^{-2}d^{2}a^{-1}$ respectively. Each of these identities contains all 4 generators and alternating positive and negative powers.\\

\noindent{}Third, if we ignore ordering and simply count powers of generators as shown in Table~\ref{table:generator_powers}, there is no simple combination that allows the possibility of eliminating negative powers. In order to remove dependence on $n$, we must make the combinations described in the previous paragraph (In Table~\ref{table:generator_powers} the rows titled ``1 and 2 Subtotal," and ``3 and 4 Subtotal"). Now to make the total power of $b$ positive or $0$, we need at least $2$ instances of the row ``3 and 4 Subtotal," but this renders the total power of $c$ negative.\\

\begin{table}[ht]
\begin{center}
\begin{tabular}{l | l | l | l | l}
Relation \hspace{10 pt} & $a$\hspace{10 pt} & $b$\hspace{10 pt} & $c$\hspace{10 pt} & $d$\hspace{10 pt} \\\hline\hline
1: $(a^{-1}b)^{10n}d^{-1}a^{2}$ & $-10n+2$ & $10n$ & $0$ & $-1$ \\\hline
2: $b^{-2}c(b^{-1}a)^{10n}$ & $10n$ & $-10n-2$ & $1$ & $0$ \\\hline
3: $(d^{-1}c)^{10n+3}c^{-1}bc^{-2}$ & $0$ & $1$ & $10n$ & $-10n-3$ \\\hline
4: $d^{2}a^{-1}d(c^{-1}d)^{10n+3}$ & $-1$ & $0$ & $-10n-3$ & $10n+6$ \\\hline
1 and 2 Subtotal & $2$ & $-2$ & $1$ & $-1$ \\\hline
3 and 4 Subtotal & $-1$ & $+1$ & $-3$ & $3$ \\\hline
1, 2, 3, and 4 Total & $1$ & $-1$ & $-2$ & $2$
\end{tabular}
\end{center}
\caption{Total powers of generators in each group relation of $G_n$.}
\label{table:generator_powers}
\end{table}

\noindent{}With these challenges in mind, we present an alternative solution in Sections \ref{section:case1} and \ref{section:case16}.\\

\noindent{}In order to show Case 1 ($a,b,c,d>1$) is not possible if $G_n$ is left-orderable, we consider sub-cases (see Table~\ref{table:case1.-}). By Proposition~\ref{proposition:non-trivial}, we know that these sub-cases represent all possibilities if $G_n$ is left-orderable. We proceed by reaching a contradiction to left-orderability in each sub-case.\\

\begin{table}[ht]
\begin{center}
\begin{tabular}{l | l | l | l }
Case \hspace{10 pt} & $a^{-1}d$\hspace{10 pt} & $b^{-1}d$\hspace{10 pt} & $b^{-1}c$\hspace{10 pt} \\\hline\hline
\case{1}{1}{} & $+$ & $+$ & $+$ \\\hline
\case{1}{2}{} & $+$ & $+$ & $-$ \\\hline
\case{1}{3}{} & $+$ & $-$ & $+$ \\\hline
\case{1}{4}{} & $+$ & $-$ & $-$ \\\hline
\case{1}{5}{} & $-$ & $+$ & $+$ \\\hline
\case{1}{6}{} & $-$ & $+$ & $-$ \\\hline
\case{1}{7}{} & $-$ & $-$ & $+$ \\\hline
\case{1}{8}{} & $-$ & $-$ & $-$ 
\end{tabular}
\end{center}
\caption{Eight sub-cases of Case 1, considering the signs of $a^{-1}d$, $b^{-1}d$, and $b^{-1}c$.}
\label{table:case1.-}
\end{table}

$\;$
\subsection{\Case{1}{2}{}-\case{1}{7}{}}

We start by showing that if $G_n$ is left-orderable, then \Cases{1}{2}{}-\case{1}{7}{} are not possible.


\begin{proposition} If $G_n$ is left-orderable, then \Cases{1}{3}{} and \case{1}{4}{} are impossible.
\end{proposition}
\begin{proof} $\;$ Suppose that $G_n$ is left-orderable, and suppose (for contradiction) that $a^{-1}d>1$ and $b^{-1}d<1$, then:
\begin{align*}
(a^{-1}d)(d^{-1}b)=a^{-1}b&>1\\
\Rightarrow{}b^{-1}a&<1,
\end{align*}
which contradicts the general assumption $b^{-1}a>1$.
\end{proof}

\begin{proposition} If $G_n$ is left-orderable, then \Cases{1}{5}{} and \case{1}{7}{} are impossible.
\end{proposition}
\begin{proof} $\;$ Suppose that $G_n$ is left-orderable, and suppose (for contradiction) that $b^{-1}c>1$ and $a^{-1}d<1$, then:
\begin{align*}
(c^{-1}b)(a^{-1}d)=b^{-1}(bc^{-1}b)a^{-1}d=b^{-1}(bc^{-1}b)(a^{-1}da^{-1})a=b^{-1}a.
\end{align*}
Where the last equality follows from Lemma~\ref{lemma:eq16}. This shows that \Cases{1}{5}{} and \case{1}{7}{} are impossible since $b^{-1}a>1$ by assumption, but $c^{-1}b<1$ and $a^{-1}d<1$ in \Cases{1}{5}{} and \case{1}{7}{}.
\end{proof}

\begin{proposition} If $G_n$ is left-orderable, then \Case{1}{2}{} is impossible.
\end{proposition}
\begin{proof} $\;$ Suppose that $G_n$ is left-orderable. By the third group relation, we have:
\begin{align}
(c^{-1}d)^{10n+3}c^2b^{-1}c&=1\nonumber{}\\
\Rightarrow{}c^2b^{-1}c(c^{-1}d)^{10n+3}&=1\nonumber{}\\
\Rightarrow{}c^2b^{-1}(dc^{-1})^{10n+3}c&=1\label{eq3:2}\\
\Rightarrow{}c^2b^{-1}(dc^{-1})^{10n+2}d&=1\nonumber{}\\
\Rightarrow{}c^2[b^{-1}(dc^{-1})b]^{10n+2}(b^{-1}d)&=1\label{eq3:3}\\
\Rightarrow{}(c^2)([b^{-1}a][a^{-1}d][c^{-1}b])^{10n+2}(b^{-1}a)(a^{-1}d)&=1.\nonumber{}
\end{align}
This shows that \Case{1}{2}{} is impossible, since $a^{-1}d>1$, $b^{-1}d>1$, $b^{-1}c<1$, and $c>1$ in \Case{1}{2}{}.
\end{proof}

\begin{proposition} If $G_n$ is left-orderable, then \Case{1}{6}{} is impossible.
\end{proposition}
\begin{proof} $\;$ Assume that $G_n$ is left-orderable. By (\ref{eq3:3}) we have:
\begin{align*}
c^2[b^{-1}(dc^{-1})b]^{10n+2}(b^{-1}d)&=1\\
\Rightarrow{}c^2[(b^{-1}d)(c^{-1}b)]^{10n+2}(b^{-1}d)&=1.
\end{align*}
This shows that \Case{1}{6}{} is impossible, since $b^{-1}d>1$, $b^{-1}c<1$, and $c>1$ in \Case{1}{6}{}.
\end{proof}

\subsection{\Case{1}{1}{}}
\noindent{}We now show that if $G_n$ is left-orderable, then \Case{1}{1}{} is impossible. To accomplish this, we consider eight new sub-cases (see Table~\ref{table:case1.i.-}).

\begin{table}[ht]
\begin{center}
\begin{tabular}{l | l | l | l }
Case\hspace{10 pt} & $ca^{-1}$\hspace{10 pt} & $da^{-1}$\hspace{10 pt} & $cb^{-1}$\hspace{10 pt} \\\hline\hline
\case{1}{1}{1} & $+$ & $+$ & $+$ \\\hline
\case{1}{1}{2} & $+$ & $+$ & $-$ \\\hline
\case{1}{1}{3} & $+$ & $-$ & $+$ \\\hline
\case{1}{1}{4} & $+$ & $-$ & $-$ \\\hline
\case{1}{1}{5} & $-$ & $+$ & $+$ \\\hline
\case{1}{1}{6} & $-$ & $+$ & $-$ \\\hline
\case{1}{1}{7} & $-$ & $-$ & $+$ \\\hline
\case{1}{1}{8} & $-$ & $-$ & $-$ 
\end{tabular}
\end{center}
\caption{Eight sub-cases of \Case{1}{1}{}, considering the signs of $ca^{-1}$, $da^{-1}$, and $cb^{-1}$.}
\label{table:case1.i.-}
\end{table}

\noindent{}Since we are working in a sub-case of \Case{1}{1}{}, we also know $a>1$, $b>1$, $c>1$, $d>1$, and the following:
\begin{align} a^{-1}d>1\label{case1.i:inEq:Ad}\\
b^{-1}d>1\label{case1.i:inEq:Bd}\\
b^{-1}c>1
\end{align}

\noindent{}As before, Proposition~\ref{proposition:non-trivial} tells us that the cases shown in Table~\ref{table:case1.i.-} represent all possibilities if $G_n$ is left-orderable.

\subsubsection{\Case{1}{1}{2}-\case{1}{1}{7}}
\noindent{}We continue by disproving left-orderability in each case. We start by showing that if $G_n$ is left-orderable, then \Cases{1}{1}{2}-\case{1}{1}{7} are impossible.

\begin{proposition} If $G_n$ is left-orderable, then \Case{1}{1}{2} is impossible.
\end{proposition}
\begin{proof} $\;$ Suppose that $G_n$ is left-orderable, and suppose (for contradiction) that $ca^{-1}>1$, $da^{-1}>1$ and $bc^{-1}>1$, then: 
\begin{align*}
(ca^{-1})(da^{-1})(bc^{-1})(bc^{-1})>1\\
\Rightarrow{}c(a^{-1}da^{-1})(bc^{-1}b)c^{-1}>1.
\end{align*}
But by Lemma~\ref{lemma:eq16}, we have:
\begin{align*}
c(a^{-1}da^{-1})(bc^{-1}b)c^{-1}=cc^{-1}=1,
\end{align*}
a contradiction.
\end{proof}

\begin{lemma} In \Case{1}{1}{}, $d^{-1}c>1$.\label{case1.1:Dc}
\end{lemma}
\begin{proof} By the fourth group relation, we have:
\begin{align}
d^{2}a^{-1}d(c^{-1}d)^{10n+3}=1\nonumber{}\\
\Rightarrow{}(d^{2})(a^{-1}d)(c^{-1}d)^{10n+3}=1.\label{case1:1:Dc}
\end{align}
Now, $d^{2}>1$ by Case 1, and $a^{-1}d>1$ by \Case{1}{1}{}. Therefore (\ref{case1:1:Dc}) shows that:
\begin{align*}
c^{-1}d&<1\\
\Rightarrow{}d^{-1}c&>1.\qedhere
\end{align*}
\end{proof}

\begin{proposition} If $G_n$ is left-orderable, then \Case{1}{1}{3} is impossible.
\end{proposition}
\begin{proof} Suppose that $G_n$ is left-orderable, and suppose (for contradiction) that $ca^{-1}>1$, $da^{-1}<1$ and $cb^{-1}>1$. Then by Lemma \ref{case1.1:Dc}, $d^{-1}c>1$ and:
\begin{align}
(ad^{-1})(d^{-1}c)(cb^{-1})(ca^{-1})(ad^{-1})&>1\nonumber{}\\
\Rightarrow{}(a)(d^{-2}c^{2})(b^{-1})(cd^{-1})&>1\nonumber{}\\
\Rightarrow{}(a)(ab^{-1})(b^{-1})(cd^{-1})=(a^{2}b^{-2})(cd^{-1})&>1,\label{proposition:case1.i.3:contradiction}
\end{align}
where the last implication follows from Lemma~\ref{lemma:eq8}, which tells us that $d^{2}a=c^{2}b$, implying $d^{-2}c^{2}=ab^{-1}$. Now by Corollary~\ref{corollary:eq6} we know:
\begin{align*}
a^{2}b^{-2}&=dc^{-1}\\
\Rightarrow{}(a^{2}b^{-2})(cd^{-1})&=1.
\end{align*}
This contradicts (\ref{proposition:case1.i.3:contradiction}).
\end{proof}

\begin{proposition} If $G_n$ is left-orderable, then \Case{1}{1}{4} is impossible.
\end{proposition}
\begin{proof} Suppose $G_n$ is left-orderable, and suppose (for contradiction) that $ca^{-1}>1$, $da^{-1}<1$ and $cb^{-1}<1$, then
\begin{align*}
[(bc^{-1})(ca^{-1})]^{10n-1}(bc^{-1})(ca^{-1})(ad^{-1})(a)&>1\\
\Rightarrow{}[(ba^{-1})]^{10n-1}(bd^{-1}a)&>1\\
\Rightarrow{}aa^{-1}(ba^{-1})^{10n-1}(bd^{-1}a)&>1\\
\Rightarrow{}aa^{-1}b(a^{-1}b)^{10n-2}a^{-1}bd^{-1}a&>1\\
\Rightarrow{}a(a^{-1}b)^{10n}d^{-1}a>1.\\
\end{align*}
This contradicts the first group relation, which says $(a^{-1}b)^{10n}d^{-1}a^{2}=1$ or equivalently $a(a^{-1}b)^{10n}d^{-1}a=1$.
\end{proof}

\begin{proposition} If $G_n$ is left-orderable, then \Case{1}{1}{5} is impossible.
\end{proposition}
\begin{proof} Suppose that $G_n$ is left-orderable, and suppose (for contradiction) that $ca^{-1}<1$, $da^{-1}>1$ and $cb^{-1}>1$, then:
\begin{align*}
(c)(cb^{-1})[(da^{-1})(ac^{-1})]^{10n+2}(d)(c^{-1}c)&>1\\
\Rightarrow{}(c^{2}b^{-1})[(dc^{-1})]^{10n+2}(dc^{-1})(c)&>1\\
\Rightarrow{}c^{2}b^{-1}(dc^{-1})^{10n+3}c&>1.
\end{align*}
This contradicts (\ref{eq3:2}), which says that
\begin{align*}
c^{2}b^{-1}(dc^{-1})^{10n+3}c&=1\qedhere
\end{align*}
\end{proof}

\begin{proposition} If $G_n$ is left-orderable, then \Case{1}{1}{6} is impossible.
\end{proposition}
\begin{proof} Suppose $G_n$ is left-orderable, and suppose (for contradiction) that $ca^{-1}<1$ and $da^{-1}>1$, then
\begin{align}
[(da^{-1})(ac^{-1})]^{10n+3}(d^{2})(da^{-1})&>1\nonumber{}\\
\Rightarrow{}(dc^{-1})^{10n+3}d^{3}a^{-1}&>1,\label{proposition:case1.i.6:contradiction}
\end{align}
but by the fourth group relation, we have:
\begin{align}
d^{2}a^{-1}d(c^{-1}d)^{10n+3}&=1\nonumber{}\\
\Rightarrow{}d^{2}a^{-1}(dc^{-1})^{10n+3}d&=1\nonumber{}\\
\Rightarrow{}(dc^{-1})^{10n+3}d^{3}a^{-1}&=1,\label{eq4:2}
\end{align}
this contradicts (\ref{proposition:case1.i.6:contradiction}).
\end{proof}

\begin{proposition} If $G_n$ is left-orderable, then \Case{1}{1}{7} is impossible.
\end{proposition}
\begin{proof} Suppose that $G_n$ is left-orderable, and suppose (for contradiction) that $ca^{-1}<1$, $da^{-1}<1$, and $cb^{-1}>1$, then: 
\begin{align*}
(cb^{-1})(cb^{-1})(ad^{-1})(ac^{-1})&>1\\
\Rightarrow{}(c)(b^{-1}cb^{-1})(ad^{-1}a)(c^{-1})&>1\\
\Rightarrow{}cc^{-1}=1&>1,
\end{align*}
where the last implication follows from Lemma~\ref{lemma:eq16}.
\end{proof}


\subsubsection{\protect\Case{1}{1}{1}}

\noindent{}We now show that if $G_n$ is left-orderable, then \Case{1}{1}{1} (see Table~\ref{table:case1.i.-}) is impossible.

\begin{lemma} In \Case{1}{1}{1}, $b^{-1}cd^{-1}b>1$.\label{case1.i.1:inEq:BcDb}
\end{lemma}
\begin{proof} By the third group relation, we have:
\begin{align}
(d^{-1}c)^{10n+3}c^{-1}bc^{-2}&=1\nonumber{}\\
\Rightarrow{}(cd^{-1})^{10n+3}bc^{-3}=(cd^{-1})^{10n+3}(bc^{-1})c^{-2}&=1\nonumber{}\\
\Rightarrow{}cd^{-1}&>1,\label{case1.i.1:inEq:cD}
\end{align}
where the last implication follows from the fact that $bc^{-1}<1$ and $c^{-1}<1$ in \Case{1}{1}{1}. Now by the third group relation, we have:
\begin{align}
(d^{-1}c)^{10n+3}c^{-1}bc^{-2}&=1\nonumber{}\\
\Rightarrow{}c^{2}b^{-1}c(c^{-1}d)^{10n+3}&=1\nonumber{}\\
\Rightarrow{}c^{2}b^{-1}cd^{-1}(dc^{-1})^{10n+3}d&=1\nonumber{}\\
\Rightarrow{}c(cb^{-1})(cd^{-1})b(b^{-1}dc^{-1}b)^{10n+3}(b^{-1}d)&=1.\label{lemmaBcDb:contradiction}
\end{align}
By (\ref{case1.i:inEq:Bd}) and (\ref{case1.i.1:inEq:cD}) it is easy to see that all expressions in parentheses in (\ref{lemmaBcDb:contradiction}) are positive except for $(b^{-1}dc^{-1}b)^{10n+3}$. This tells us that:
\begin{align*}
b^{-1}dc^{-1}b&<1\\
\Rightarrow{}b^{-1}cd^{-1}b&>1.\qedhere
\end{align*}
\end{proof}

\begin{lemma} In \Case{1}{1}{1}, $d^{-1}c^{-1}d^{2}>1$.\label{lemma:inEq:DCdd}
\end{lemma}
\begin{proof} $\;$ By (\ref{eq4:2}), we have:
\begin{align*}
(d^{-1}a)d^{-2}(d^{-1}c)^{10n+3}=1.
\end{align*}
However, by (\ref{case1.i:inEq:Ad}), $d^{-1}a<1$ (as is $d^{-2}$), so (\ref{eq4:2}) shows that:
\begin{align}
d^{-1}c>1.\label{case1.i:inEq:Dc}
\end{align}
Now consider:
\begin{align}
(d^{-2}cd)(d^{-1}c)(b^{-1}cd^{-1}b)&=d^{-2}c^{2}b^{-1}cd^{-1}b<1,\label{lemma:DCdd:contradiction}
\end{align}
where the last inequality follows from Lemma~\ref{lemma:inEq7}. By (\ref{case1.i:inEq:Dc}) and Lemma~\ref{case1.i.1:inEq:BcDb} we see that $(d^{-1}c)(b^{-1}cd^{-1}b)>1$, therefore (\ref{lemma:DCdd:contradiction}) shows that:
\begin{align*}
d^{-2}cd&<1\\
\Rightarrow{}d^{-1}c^{-1}d^{2}&>1.\qedhere
\end{align*}
\end{proof}

\begin{corollary} In \Case{1}{1}{1}, $d^{-1}c^{-1}dc>1$ and $b^{-1}c^{-1}d^{2}>1$.\label{corollary:inEq:DCdd}
\end{corollary}
\begin{proof} $\;$ These are immediate consequences of Lemma~\ref{lemma:inEq:DCdd} since $d^{-1}c>1$ in \Case{1}{1}{} (by (\ref{case1.i:inEq:Dc})) and $b^{-1}d>1$ in \Case{1}{1}{} (by (\ref{case1.i:inEq:Bd})).
\end{proof}

\begin{proposition} If $G_n$ is left-orderable, then \Case{1}{1}{1} ($ca^{-1}>1$, $da^{-1}>1$, and $cb^{-1}>1$) is impossible.
\end{proposition}
\begin{proof} $\;$ Suppose $G_n$ is left-orderable, and suppose (for contradiction) that $ca^{-1}>1$, $da^{-1}>1$, and $cb^{-1}>1$. By the third group relation we have:
\begin{align}
(d^{-1}c)^{10n+3}c^{-1}bc^{-2}&=1\nonumber{}\\
\Rightarrow{}cb^{-1}c(c^{-1}d)^{10n+3}c&=1\nonumber{}\\
\Rightarrow{}c(b^{-1}d)(c^{-1}d)^{10n+2}c&=1\nonumber{}\\
\Rightarrow{}(cb^{-1})d(c^{-1}d)(c^{-1}d)^{10n}(c^{-1}d)c&=1\nonumber{}\\
\Rightarrow{}(cb^{-1})d(bb^{-1})c^{-1}d(dd^{-1})(c^{-1}d)^{10n}(dd^{-1})c^{-1}dc&=1\nonumber{}\\
\Rightarrow{}(cb^{-1})(db)(b^{-1}c^{-1}d^{2})d^{-1}(c^{-1}d)^{10n}d(d^{-1}c^{-1}dc)&=1\nonumber{}\\
\Rightarrow{}(cb^{-1})(db)(b^{-1}c^{-1}d^{2})(d^{-1}c^{-1}d^{2})^{10n}(d^{-1}c^{-1}dc)&=1.\label{proposition:1.i.1:contradiction}
\end{align}
Now $cb^{-1}>1$ by assumption in \Case{1}{1}{1}. Similarly, $d>1$ and $b>1$ by assumption in Case 1, thus $db>1$. The remaining terms in parentheses in (\ref{proposition:1.i.1:contradiction}) are positive by Lemma~\ref{lemma:inEq:DCdd} and Corollary~\ref{corollary:inEq:DCdd}. We have therefore reached a contradiction, proving that if $G_n$ is left-orderable, then \Case{1}{1}{1} is impossible.
\end{proof}


\subsubsection{\Case{1}{1}{8}}

\noindent{}Next we show that if $G_n$ is left-orderable, then \Case{1}{1}{8} is impossible. After Proposition~\ref{proposition:case1.i.8}, all sub-cases of \Case{1}{1}{} will have been eliminated, showing that \Case{1}{1}{} is impossible if $G_n$ is left-orderable.

\begin{lemma} In \Case{1}{1}{8}, $ab^{-1}>1$.
\label{eq8iL}
\end{lemma}
\begin{proof} By the second group relation, we have:
\begin{align*}
b^{-2}c(b^{-1}a)^{10n}&=1\\
\Rightarrow{}(b^{-1})(cb^{-1})(ab^{-1})^{10n}&=1\\
\Rightarrow{}ab^{-1}&>1,
\end{align*}
where the last implication follows from $b^{-1}<1$ (in Case 1), and $cb^{-1}<1$ (in \Case{1}{1}{8}).
\end{proof}

\begin{proposition} If $G_n$ is left-orderable, then \Case{1}{1}{8} ($ca^{-1}<1$, $da^{-1}<1$, and $cb^{-1}<1$) is impossible.\label{proposition:case1.i.8}
\end{proposition}
\begin{proof} $\;$Suppose $G_n$ is left-orderable and suppose (for contradiction), that $ca^{-1}<1$, $da^{-1}<1$, and $cb^{-1}<1$. By Corollary~\ref{corollary:eq6}, we have:
\begin{align}
a^{2}b^{-2}&=dc^{-1}\nonumber{}\\
\Rightarrow{}(b^{-1})(cd^{-1})(a^{2}b^{-1})&=1\nonumber{}\\
\Rightarrow{}(b^{-1}a)(a^{-1}c)(d^{-1}a)(ab^{-1})&=1\nonumber{}\\
\Rightarrow{}(a^{-1}c)(d^{-1}a)&<1,\label{eqproposition8I}
\end{align}
where the last implication follows from the general assumption $(b^{-1}a)>1$, and since $ab^{-1}>1$ by Lemma~\ref{eq8iL}. Nevertheless, by (\ref{eq3:2}):
\begin{align*}
(c^{2}b^{-1})(dc^{-1})^{10n+3}(c)&=1\\
\Rightarrow{}(c^{2})(b^{-1}a)(a^{-1}dc^{-1}a)^{10n+2}(a^{-1}d)(c^{-1}c)&=1\\
\Rightarrow{}(a^{-1}d)(c^{-1}a)&<1,
\end{align*}
where the last implication follows from $b^{-1}a>1$, $c^{2}>1$ (in Case 1), and $a^{-1}d>1$ (in \Case{1}{1}{}), i.e. $(a^{-1}c)(d^{-1}a)>1$, which contradicts (\ref{eqproposition8I}). Therefore if $G_n$ is left-orderable \Case{1}{1}{8} is impossible.
\end{proof}


\subsection{\Case{1}{8}{}}

\noindent{}We will now show that if $G_n$ is left-orderable, then \Case{1}{8}{} ($d^{-1}a>1$, $d^{-1}b>1$, $c^{-1}b>1$)is impossible.

\begin{lemma} In \Case{1}{8}{}, $c^{-1}d > 1$.
\label{lemma:inEq:Cd}
\end{lemma}

\begin{proof} Starting from the first group relation, we have:
\begin{align}
(a^{-1}b)^{10n}d^{-1}a^{2}&=1\nonumber{}\\
\Rightarrow{}a(a^{-1}b)^{10n}d^{-1}a&=1\nonumber{}\\
\Rightarrow{}(ba^{-1})^{10n-1}bd^{-1}a&=1\label{eq3:4}\\
\Rightarrow{}(ba^{-1})^{10n-2}ba^{-1}bd^{-1}a&=1\nonumber{}\\
\Rightarrow{}(cc^{-1})(ba^{-1}(dd^{-1}))^{10n-2}(cc^{-1})ba^{-1}(dd^{-1})bd^{-1}a&=1\nonumber{}\\
\Rightarrow{}c([c^{-1}ba^{-1}d][d^{-1}c])^{10n-2}(c^{-1}ba^{-1}d)(d^{-1}b)(d^{-1}a)&=1,\label{lemma:inEq:Cd:contradiction}
\end{align}
but $d^{-1}b>1$, $d^{-1}a>1$, and $c>1$ in \Case{1}{8}{}. Further, we know by Lemma~\ref{lemma:inEq:CbAd} that $c^{-1}ba^{-1}d>1$. Therefore, (\ref{lemma:inEq:Cd:contradiction}) shows that:
\begin{align*}
d^{-1}c&<1\\
\Rightarrow{}c^{-1}d&>1.\qedhere
\end{align*}
\end{proof}

\begin{lemma} In \Case{1}{8}{}, $ab^{-1} > 1$.
\label{lemma:inEq:aB}
\end{lemma}

\begin{proof} By (\ref{eq3:4}):
\begin{align}
(ba^{-1})^{10n-1}b(a^{-1}a)(d^{-1}a)&=1\nonumber{}\\
(ba^{-1})^{10n}a(d^{-1}a)&=1,\label{lemma:inEq:aB:contradiction}
\end{align}
but $d^{-1}a>1$ in \Case{1}{8}{} and $a>1$ in Case 1, so (\ref{lemma:inEq:aB:contradiction}) shows that:
\begin{align*}
ba^{-1}&<1\\
\Rightarrow{}ab^{-1}&>1.\qedhere
\end{align*}
\end{proof}

\begin{lemma} In \Case{1}{8}{}, $a^{-1}dc^{-1}a > 1$. \label{lemma:inEq:AdCa}
\end{lemma}
\begin{proof} By Corollary~\ref{corollary:eq6}, we have:
\begin{align}
a^{2}b^{-2}&=dc^{-1}\nonumber{}\\
\Rightarrow{}b^{-2}cd^{-1}a^{2} &= 1\nonumber{}\\
\Rightarrow{}b^{-1}cd^{-1}a^{2}b^{-1} &= 1\nonumber{}\\
\Rightarrow{}(b^{-1}a)(a^{-1}cd^{-1}a)(ab^{-1}) &= 1.\label{lemma:inEq:AdCa:contradiction}
\end{align}
But $b^{-1}a>1$ by assumption and $ab^{-1}>1$ in \Case{1}{8}{} by Lemma~\ref{lemma:inEq:aB} so (\ref{lemma:inEq:AdCa:contradiction}) shows that:
\begin{align*}
a^{-1}cd^{-1}a&<1\\
\Rightarrow{}a^{-1}dc^{-1}a&>1.\qedhere
\end{align*}
\end{proof}

\begin{lemma} In \Case{1}{8}{}, $c^{-1}d^{-2}a>1$.
\label{lemma:inEq:CDDa}
\end{lemma}
\begin{proof} Starting from the fourth group relation:
\begin{align}
d^{2}a^{-1}d(c^{-1}d)^{10n+3} &= 1\nonumber{}\\
\Rightarrow{}a^{-1}d(c^{-1}d)^{10n+3}d^{2} &= 1\nonumber{}\\
\Rightarrow{}a^{-1}(dc^{-1})^{10n+3}d^{3}&=1\nonumber{}\\
\Rightarrow{}a^{-1}(dc^{-1})^{10n+3}(aa^{-1})d^{2}(cc^{-1})d&=1\nonumber{}\\
\Rightarrow{}(a^{-1}dc^{-1}a)^{10n+3}(a^{-1}d^{2}c)(c^{-1}d)&=1.\label{lemma:inEq:CDDa:contradiction}
\end{align}
But $a^{-1}dc^{-1}a>1$ in \Case{1}{8}{} by Lemma~\ref{lemma:inEq:AdCa} and $c^{-1}d>1$ in \Case{1}{8}{} by Lemma~\ref{lemma:inEq:Cd}, so (\ref{lemma:inEq:CDDa:contradiction}) shows that:
\begin{align*}
a^{-1}d^{2}c&<1\\
\Rightarrow{}c^{-1}d^{-2}a&>1.\qedhere
\end{align*}
\end{proof}

\begin{lemma} In \Case{1}{8}{}, $d^{-1}ad^{-1}>1$.
\label{lemma:inEq:DaD}
\end{lemma}
\begin{proof}
Starting from the fourth group relation, we have:
\begin{align}
(c^{-1}d)^{10n+3}d^{2}a^{-1}d&=1\nonumber{}\\
\Rightarrow{}(c^{-1}d)^{10n+3}(cc^{-1})d^{2}a^{-1}d&=1\nonumber{}\\
\Rightarrow{}(c^{-1}d)^{10n+3}(c)(c^{-1}d)(da^{-1}d)&=1.\label{lemma:inEq:DaD:contradiction}
\end{align}
We know that $c^{-1}d>1$ in \Case{1}{8}{} by Lemma~\ref{lemma:inEq:Cd} and we know that $c>1$ in \Case{1}{8}{}, so (\ref{lemma:inEq:DaD:contradiction}) shows that:
\begin{align*}
da^{-1}d&<1\\
\Rightarrow{}d^{-1}ad^{-1}&>1.\qedhere
\end{align*}
\end{proof}

\begin{lemma} In \Case{1}{8}{}, $cba^{-1}>1$.
\label{lemma:inEq:cbA}
\end{lemma}

\begin{proof} By Lemma~\ref{lemma:eq8}, we have:
\begin{align}
c^{2}b&=d^{2}a\nonumber{}\\
\Rightarrow{}c^{-1}d^{2}ab^{-1}c^{-1}&= 1\nonumber{}\\
\Rightarrow{}(c^{-1}d)d(ab^{-1}c^{-1})&=1.\label{lemma:inEq:cbA:contradiction}
\end{align}
We know that $c^{-1}d>1$ in \Case{1}{8}{} by Lemma~\ref{lemma:inEq:Cd} and we know that $d>1$ in \Case{1}{8}{}, so (\ref{lemma:inEq:cbA:contradiction}) shows that:
\begin{align*}
ab^{-1}c^{-1}&<1\\
\Rightarrow{}cba^{-1}&>1.\qedhere
\end{align*}
\end{proof}

\begin{lemma} In \Case{1}{8}{}, $da^{-1}bc^{-1}>1$.
\label{lemma:inEq:dAbC}
\end{lemma}
\begin{proof} By Lemma~\ref{lemma:eq16}, we have:
\begin{align*}
b^{-1}cb^{-1}ad^{-1}a &= 1\\
\Rightarrow{}(cb^{-1}ad^{-1})(ab^{-1}) &=1\\
\Rightarrow{}ab^{-1}&=da^{-1}bc^{-1}\\
\Rightarrow{}da^{-1}bc^{-1}&>1,
\end{align*}
where the last implication follows from Lemma~\ref{lemma:inEq:aB}.
\end{proof}

\begin{lemma} In \Case{1}{8}{}, $cb^{-1}ac^{-1}>1$.
\label{lemma:inEq:cBaC}
\end{lemma}
\begin{proof} By (\ref{eq3:4}), we have:
\begin{align}
(ba^{-1})^{10n}bd^{-1}a&=1\nonumber{}\\
\Rightarrow{}b(a^{-1}b)^{10n}d^{-1}a&=1\nonumber{}\\
\Rightarrow{}b(a^{-1}b)(a^{-1}b)^{10n-1}d^{-1}a&=1\nonumber{}\\
\Rightarrow{}(ba^{-1})(bc^{-1})(ca^{-1}bc^{-1})^{10n-1}cd^{-1}a&=1\nonumber{}\\
\Rightarrow{}(cc^{-1})(ba^{-1})(dd^{-1})(ad^{-1}da^{-1})(bc^{-1})(ca^{-1}bc^{-1})^{10n-1}cd^{-1}a&=1\nonumber{}\\
\Rightarrow{}c(c^{-1}ba^{-1}d)(d^{-1}ad^{-1})(da^{-1}bc^{-1})(ca^{-1}bc^{-1})^{10n-1}c(d^{-1}a)&=1.\label{lemma:inEq:cBaC:contradiction}
\end{align}
We know $c^{-1}ba^{-1}d>1$ in \Case{1}{8}{} by Lemma~\ref{lemma:inEq:CbAd}; we know $d^{-1}ad^{-1}>1$ in \Case{1}{8}{} by Lemma~\ref{lemma:inEq:DaD}; and we know $da^{-1}bc^{-1}>1$ by Lemma~\ref{lemma:inEq:dAbC}. Furthermore, we know $d^{-1}a>1$ and $c>1$ by assumption in \Case{1}{8}{}. Therefore, (\ref{lemma:inEq:cBaC:contradiction}) shows that:
\begin{align*}
ca^{-1}bc^{-1}&<1\\
\Rightarrow{}cb^{-1}ac^{-1}&>1.\qedhere
\end{align*}
\end{proof}

\begin{lemma} In \Case{1}{8}{}, $cba^{-1}c^{-1}>1$.
\label{lemma:inEq:cbAC}
\end{lemma}
\begin{proof} By Lemma~\ref{lemma:eq8}, we have:
\begin{align}
c^{2}b&=d^{2}a\nonumber{}\\
\Rightarrow{}c^{-1}d^{2}ab^{-1}c^{-1}&=1\nonumber{}\\
\Rightarrow{}c^{-1}d^{2}a^{-1}a^{2}b^{-1}c^{-1}&=1\nonumber{}\\
\Rightarrow{}c^{-1}d^{2}a^{-1}(bc^{-1}cb^{-1})a(c^{-1}c)ab^{-1}c^{-1}&=1\nonumber{}\\
\Rightarrow{}(c^{-1}d)(da^{-1}bc^{-1})(cb^{-1}ac^{-1})(cab^{-1}c^{-1})&=1.\label{lemma:inEq:cbAC:contradiction}
\end{align}
We know that $c^{-1}d>1$ in \Case{1}{8}{} by Lemma~\ref{lemma:inEq:Cd}; we know that $da^{-1}bc^{-1}>1$ in \Case{1}{8}{} by Lemma~\ref{lemma:inEq:dAbC}; and we know that $cb^{-1}ac^{-1}>1$ in \Case{1}{8}{} by Lemma~\ref{lemma:inEq:cBaC}. Therefore, (\ref{lemma:inEq:cbAC:contradiction}) shows that:
\begin{align*}
cab^{-1}c^{-1}&<1\\
\Rightarrow{}cba^{-1}c^{-1}&>1.\qedhere
\end{align*}
\end{proof}

\begin{proposition}
If $G_{n}$ is left-orderable then \Case{1}{8}{} ($d^{-1}a > 1$, $d^{-1}b>1$, and $c^{-1}b>1$) is not possible.
\label{proposition:case1.viii}
\end{proposition}
\begin{proof} By (\ref{eq3:4}), we have:
\begin{align}
(ba^{-1})^{10n}bd^{-1}a&=1\nonumber{}\\
\Rightarrow{}(ba^{-1})(bd^{-1}a)(ba^{-1})^{10n-2}(ba^{-1})&=1\nonumber{}\\
\Rightarrow{}(ba^{-1})(bd^{-1}a)(c^{-1}c)(ba^{-1})^{10n-2}(c^{-1}c)(ba^{-1})&=1\nonumber{}\\
\Rightarrow{}(ba^{-1})(bd^{-1})(ac^{-1})(cba^{-1}c^{-1})^{10n-2}(c)(ba^{-1})&=1\nonumber{}\\
\Rightarrow{}(cc^{-1})(ba^{-1})(dd^{-1})(bd^{-1})(ad^{-1}da^{-1})(bc^{-1}cb^{-1})\nonumber{}\\
(ac^{-1})(cba^{-1}c^{-1})^{10n-1}(cba^{-1})&=1\nonumber{}\\
\Rightarrow{}c(c^{-1}ba^{-1}d)(d^{-1}b)(d^{-1}ad^{-1})(da^{-1}bc^{-1})\nonumber{}\\
(cb^{-1}ac^{-1})(cba^{-1}c^{-1})^{10n-1}(cba^{-1})&=1.\label{proposition:case1.viii:contradiction}
\end{align}
We know that $c^{-1}ba^{-1}d$, $d^{-1}ad^{-1}$, $da^{-1}bc^{-1}$, $cb^{-1}ac^{-1}$, $cba^{-1}c^{-1}$, and  $cba^{-1}$ are all positive in \Case{1}{8}{} by Lemmas~\ref{lemma:inEq:CbAd}, \ref{lemma:inEq:DaD}, \ref{lemma:inEq:dAbC}, \ref{lemma:inEq:cBaC}, \ref{lemma:inEq:cbAC}, and \ref{lemma:inEq:cbA} respectively. Also, $c$ is positive in \Case{1}{8}{} by assumption. Therefore (\ref{proposition:case1.viii:contradiction}) shows that if $G_n$ is left-orderable, then \Case{1}{8}{} is not possible.
\end{proof}



\noindent{}With Proposition~\ref{proposition:case1.viii}, we have eliminated the one remaining sub-case of Case 1. Thus, we have shown that if $G_n$ is left-orderable, then the only option for the signs of the four generators is Case 16. That is, if $G_n$ is left-orderable then $a<1$, $b<1$, $c<1$, and $d<1$.

%% file: body/case_16.tex
\section{Case 16}
\label{section:case16}

\noindent{}We will now show that if $G_n$ is left-orderable, then Case 16 (see Table~\ref{table:case16}) is not possible.

\begin{table}[ht]
\begin{center}
\begin{tabular}{l | l | l | l | l}
Case\hspace{10 pt} & $a$\hspace{10 pt} & $b$\hspace{10 pt} & $c$\hspace{10 pt} & $d$\hspace{10 pt} \\\hline\hline
16 & $-$ & $-$ & $-$ & $-$
\end{tabular}
\end{center}
\caption{The signs of the 4 generators in Case 16.}
\label{table:case16}
\end{table}

\noindent{}We start by proving the signs of a few key elements.

\subsection{Inequalities for Case 16}

\begin{lemma} In Case 16, $d^{-1}a>1$.
\label{lemma:case16:Da}
\end{lemma}
\begin{proof}
By (\ref{eq1:2}):
\begin{align*}
d^{-1}a^{2}=(b^{-1}a)^{10n} \\
\Rightarrow{}d^{-1}a=(b^{-1}a)^{10n}a^{-1}.
\end{align*}
Thus, $d^{-1}a$ is a product of positive elements.\qedhere
\end{proof}

\begin{lemma} In Case 16, $c^{-1}b>1$.
\label{lemma:case16:Cb}
\end{lemma}
\begin{proof}
By (\ref{eq2:2}):
\begin{align*}
c^{-1}b^{2}=(b^{-1}a)^{10n}\\
\Rightarrow{}c^{-1}b=(b^{-1}a)^{10n}b^{-1}.
\end{align*}
Thus, $c^{-1}b$ is a product of positive elements.
\end{proof}

\begin{lemma} In Case 16, $c^{-1}d>1$.
\label{lemma:case16:Cd}
\end{lemma}
\begin{proof} Suppose (for contradiction) that $c^{-1}d<1$, or equivalently that $d^{-1}c>1$. By the third group relation, we have:
\begin{align}
(d^{-1}c)^{10n+3}c^{-1}bc^{-2}&=1\nonumber{}\\
\Rightarrow{}(d^{-1}c)^{10n+3}&=c^{2}b^{-1}c.\label{lemma:case16:Cd:contradiction}
\end{align}
Since we are assuming $d^{-1}c>1$, (\ref{lemma:case16:Cd:contradiction}) shows that $c^{2}b^{-1}c>1$. By (\ref{eq2:2}), $c^{-1}b^{2}=(b^{-1}a)^{10n}>1$, thus:
\begin{align*}
(c^{2}b^{-1}c)(c^{-1}b^{2})>1\\
\Rightarrow{}c^{2}b>1.
\end{align*}
This is a contradiction, since both $b$ and $c$ are negative in Case 16. Therefore, $c^{-1}d\geq{}1$. However, by Proposition~\ref{proposition:non-trivial}, $c\neq{}d$ and thus $c^{-1}d\neq{}1$. Therefore, $c^{-1}d>1$.
\end{proof}

\begin{corollary} In Case 16, $c^{-1}a>1$.
\label{lemma:case16:Ca}
\end{corollary}
\begin{proof}
This follows from Lemma~\ref{lemma:case16:Da} and Lemma~\ref{lemma:case16:Cd} since:
\begin{align*}
c^{-1}a&=(c^{-1}d)(d^{-1}a).\qedhere
\end{align*}
\end{proof}

\begin{lemma} In Case 16, $ba^{-1}>1$.
\label{lemma:case16:bA}
\end{lemma}
\begin{proof}
Starting from the fourth group relation, we have:
\begin{align}
d^{2}a^{-1}d(c^{-1}d)^{10n+3}&=1\nonumber{}\\
\Rightarrow{}d^{2}a^{-1}d^{2}(d^{-1}(cc^{-1})c^{-1}d^{2})^{10n+3}d^{-1}&=1\nonumber{}\\
\Rightarrow{}da^{-1}d^{2}([d^{-1}c][c^{-2}d^{2}])^{10n+3}&=1\nonumber{}\\
\Rightarrow{}da^{-1}d^{2}([d^{-1}c][ba^{-1}])^{10n+3}&=1,\label{lemma:case16:bA:contradiction}
\end{align}
where the last implication follows from Lemma \ref{lemma:eq8}. Now $a^{-1}d<1$ by Lemma~\ref{lemma:case16:Da}, so we know $da^{-1}d^{2}<1$. Therefore, (\ref{lemma:case16:bA:contradiction}) tells us that:
\begin{align*}
([d^{-1}c][ba^{-1}])^{10n+3}>1,
\end{align*}
but $d^{-1}c<1$ by Lemma~\ref{lemma:case16:Cd}, so we must have:
\begin{align*}
ba^{-1}&>1.\qedhere
\end{align*}
\end{proof}

\subsection{Concordance of signs of a few useful elements}

Next we will show that in Case 16, $bc^{-1}$, $dc^{-1}$, $ad^{-1}$, $ac^{-1}$, $bc^{-3}$, and $ad^{-3}$ must all have the same sign. We begin by showing that all of these elements are non-trivial.

\begin{proposition}
In Case 16, $bc^{-1}\neq1$, $dc^{-1}\neq1$, $ad^{-1}\neq1$, $ac^{-1}\neq1$, $bc^{-3}\neq1$, and $ad^{-3}\neq1$.
\end{proposition}
\begin{proof}
Consequence of Proposition~\ref{proposition:non-trivial} and Corollary~\ref{corollary:non-trivial2}.
\end{proof}

\noindent{}Now that we know that each of these elements can only be positive or negative, we proceed to show that they all have the same sign.

\begin{lemma} In Case 16, $ad^{-1}<1$ implies $dc^{-1}<1$.
\label{lemma:iffs:1}
\end{lemma}
\begin{proof}
Starting from the second group relation, we have:
\begin{align}
b^{-2}c(b^{-1}a)^{10n}&=1\nonumber{}\\
\Rightarrow{}b^{-1}c(b^{-1}a)^{10n}b^{-1}&=1\nonumber{}\\
\Rightarrow{}b^{-1}cb^{-1}(ab^{-1})^{10n}&=1\nonumber{}\\
\Rightarrow{}b^{-1}cb^{-1}a^{-1}(a^{2}b^{-1}a^{-1})^{10n}a&=1\nonumber{}\\
\Rightarrow{}ab^{-1}cb^{-1}a^{-1}((a^{2}b^{-2})(ba^{-1}))^{10n}&=1\nonumber{}\\
\Rightarrow{}a(b^{-1}cb^{-1})a^{-1}([dc^{-1}][ba^{-1}])^{10n}&=1\nonumber{}\\
\Rightarrow{}da^{-2}([dc^{-1}][ba^{-1}])^{10n}&=1.\label{lemma:iffs:1:contradiction}
\end{align}
Where the second to last implication follows from Corollary~\ref{corollary:eq6} and last implication follows from Lemma~\ref{lemma:eq16}. Now if $ad^{-1}<1$ we have:
\begin{align*}
da^{-1}&>1\\
\Rightarrow{}da^{-2}&>1.
\end{align*}
Thus if $ad^{-1}<1$, (\ref{lemma:iffs:1:contradiction}) implies:
\begin{align*}
(dc^{-1})(ba^{-1})<1,
\end{align*}
which shows that $dc^{-1}<1$, since $ba^{-1}>1$ by Lemma~\ref{lemma:case16:bA}.
\end{proof}

\begin{lemma} In Case 16, $dc^{-1}>1$ implies $ac^{-1}>1$.
\label{lemma:iffs:2.1}
\end{lemma}
\begin{proof}
By Lemma~\ref{lemma:iffs:1}:
\begin{align*}
dc^{-1}>1 \Rightarrow{} ad^{-1}>1.
\end{align*}
This completes the proof since:
\begin{align*}
ac^{-1}&=(ad^{-1})(dc^{-1}).\qedhere
\end{align*}
\end{proof}

\begin{lemma} In Case 16, $ad^{-1}<1$ implies $ac^{-1}<1$.
\label{lemma:iffs:2.2}
\end{lemma}
\begin{proof}
By Lemma~\ref{lemma:iffs:1}:
\begin{align*}
ad^{-1}<1 \Rightarrow{} dc^{-1}<1.
\end{align*}
This completes the proof since:
\begin{align*}
ac^{-1}&=(ad^{-1})(dc^{-1}).\qedhere
\end{align*}
\end{proof}

\begin{corollary} In Case 16, $dc^{-1}>1$ if and only if $ad^{-1}>1$ if and only if $ac^{-1}>1$.
\label{lemma:iffs:3}
\end{corollary}
\begin{proof}
Starting from the fourth group relation, we have:
\begin{align}
d^{2}a^{-1}d(c^{-1}d)^{10n+3}&=1\nonumber{}\\
\Rightarrow{}a^{-1}d(c^{-1}d)^{10n+3}d^{2}&=1\nonumber{}\\
\Rightarrow{}a^{-1}(dc^{-1})^{10n+3}d^{3}&=1\nonumber{}\\
\Rightarrow{}d^{2}(da^{-1})(dc^{-1})^{10n+3}&=1.\label{lemma:iffs:3:contradiction}
\end{align}
Now $d<1$ in Case 16, so (\ref{lemma:iffs:3:contradiction}) shows that:
\begin{align*}
ad^{-1}>1\Rightarrow{}da^{-1}<1\Rightarrow{}dc^{-1}>1.
\end{align*}
In conjunction with Lemma~\ref{lemma:iffs:1}, this shows that $dc^{-1}>1$ if and only if $ad^{-1}>1$. Now by Lemma~\ref{lemma:iffs:2.1} we have:
\begin{align*}
ad^{-1}>1\Rightarrow{}dc^{-1}>1\Rightarrow{}ac^{-1}>1,
\end{align*}
and by Lemma~\ref{lemma:iffs:2.2} we have:
\begin{align*}
ad^{-1}<1\Rightarrow{}ac^{-1}<1.
\end{align*}
Thus, $ad^{-1}>1$ if and only if $ac^{-1}>1$.
\end{proof}

\begin{lemma} In Case 16, $dc^{-1}>1$ if and only if $bc^{-1}>1$.
\label{lemma:iffs:4}
\end{lemma}
\begin{proof}[Proof of reverse direction]
Starting from the third group relation, we have:
\begin{align*}
(d^{-1}c)^{10n+3}c^{-1}bc^{-2}&=1\\
\Rightarrow{}c^{-1}(cd^{-1})^{10n+3}bc^{-2}&=1\\
\Rightarrow{}(cd^{-1})^{10n+3}(bc^{-2})&=c<1,
\end{align*}
and so:
\begin{align*}
dc^{-1}<1\Rightarrow{}cd^{-1}>1\Rightarrow{}bc^{-2}&<1\Rightarrow{}bc^{-1}<1. \qedhere
\end{align*}
\end{proof}

\begin{proof}[Proof of forward direction]
First note that:
\begin{align*}
bc^{-1}=(ba^{-1})(ac^{-1}).
\end{align*}
But $ba^{-1}>1$ by Lemma~\ref{lemma:case16:bA} so this shows that:
\begin{align*}
dc^{-1}>1\Rightarrow{}ac^{-1}>1\Rightarrow{}bc^{-1}>1,
\end{align*}
where the first implication follows from Corollary~\ref{lemma:iffs:3}.
\end{proof}

\begin{lemma} In Case 16, $dc^{-1}>1$ if and only if $bc^{-3}>1$ if and only if $ad^{-3}>1$.
\label{lemma:iffs:5}
\end{lemma}
\begin{proof}
By the third group relation, we have:
\begin{align*}
(d^{-1}c)^{10n+3}c^{-1}bc^{-2}&=1\\
\Rightarrow{}(cd^{-1})^{10n+3}bc^{-3}&=1\\
\Rightarrow{}bc^{-3}&=(dc^{-1})^{10n+3}.
\end{align*}
Thus $bc^{-3}>1$ must have the same sign as $dc^{-1}$. By the fourth group relation, we have:
\begin{align*}
(d^{-1}c)^{10n+3}d^{-1}ad^{-2}&=1\\
\Rightarrow{}(cd^{-1})^{10n+3}ad^{-3}&=1\\
\Rightarrow{}ad^{-3}&=(dc^{-1})^{10n+3}.
\end{align*}
Thus $ad^{-3}>1$ must have the same sign as $dc^{-1}$.
\end{proof}

\begin{proposition} In Case 16, the following elements all have the same sign: $dc^{-1}$, $ad^{-1}$, $ac^{-1}$, $bc^{-1}$, $bc^{-3}$, and $ad^{-3}$.
\label{proposition:iffs}
\end{proposition}
\begin{proof}
The proposition is evident by combining Corollary~\ref{lemma:iffs:3}, Lemma~\ref{lemma:iffs:4}, and Lemma~\ref{lemma:iffs:5}.
\end{proof}

\noindent{}In order to show Case 16 ($a,b,c,d<1$) is not possible if $G_n$ is left-orderable, we consider sub-cases (see Table~\ref{table:cases16b}). Because of Proposition~\ref{proposition:iffs} , it is easy to see that there are only two possible sub-cases of Case 16 considering the signs of $dc^{-1}$, $ad^{-1}$, $ac^{-1}$, $bc^{-1}$, $bc^{-3}$, and $ad^{-3}$

\begin{table}[ht]
\begin{center}
\begin{tabular}{l | l | l | l | l | l | l}
Case \hspace{10 pt} & $dc^{-1}$\hspace{10 pt} & $ad^{-1}$\hspace{10 pt} & $ac^{-1}$\hspace{10 pt} & $bc^{-1}$\hspace{10 pt} & $bc^{-3}$\hspace{10 pt} & $ad^{-3}$\hspace{10 pt}  \\\hline\hline
\case{16}{1}{} & $+$ & $+$ & $+$ & $+$ & $+$ & $+$ \\\hline
\case{16}{2}{} & $-$ & $-$ & $-$ & $-$ & $-$ & $-$
\end{tabular}
\end{center}
\caption{The two possible sub-cases of Case 16 considering the signs of $dc^{-1}$, $ad^{-1}$, $ac^{-1}$, $bc^{-1}$, $bc^{-3}$, and $ad^{-3}$}
\label{table:cases16b}
\end{table}

\subsection{\Case{16}{1}{}}

As a reminder, since we are working in a sub-case of Case 16, we know $a<1$, $b<1$, $c<1$, and $d<1$.

\begin{lemma} In \Case{16}{1}{}, $a^{-1}dc^{-1}a>1$.
\label{lemma:case16.A:AdCa}
\end{lemma}
\begin{proof} By the third group relation, we have:
\begin{align}
(d^{-1}c)^{10n+3}c^{-1}bc^{-2}&=1\nonumber{}\\
\Rightarrow{}c^{-1}(cd^{-1})^{10n+3}bc^{-2}&=1\nonumber{}\\
\Rightarrow{}c^{-1}(aa^{-1})(cd^{-1})^{10n+3}(aa^{-1})b(a^{-1}a)c^{-1}(a^{-1}a)c^{-1}&=1\nonumber{}\\
\Rightarrow{}(c^{-1}a)(a^{-1}cd^{-1}a)^{10n+3}(a^{-1})(ba^{-1})(ac^{-1})(a^{-1})(ac^{-1})&=1.\label{16.A:AdCa}\
\end{align}
But $c^{-1}a>1$ by Corollary \ref{lemma:case16:Ca} , $a^{-1}>1$ in Case 16, $ba^{-1}>1$ by Lemma \ref{lemma:case16:bA} and $ac^{-1}>1$ in \Case{16}{1}{} by assumption, so (\ref{16.A:AdCa}) shows that:
\begin{align*}
a^{-1}cd^{-1}a&<1\\
\Rightarrow{}a^{-1}dc^{-1}a&>1.\qedhere
\end{align*}
\end{proof}

\begin{lemma} In \Case{16}{1}{}, $ab^{-2}a>1$.
\label{lemma:case16.A:aBBa}
\end{lemma}
\begin{proof} By Lemma \ref{lemma:eq5}, we have:
\begin{align}
dc^{-1}b^{2}a^{-2}=1\nonumber{}\\
\Rightarrow{}a^{-1}dc^{-1}b^{2}a^{-1}=1\nonumber{}\\
\Rightarrow{}a^{-1}dc^{-1}(aa^{-1})b^{2}a^{-1}=1\nonumber{}\\
\Rightarrow{}(a^{-1}dc^{-1}a)(a^{-1}b^{2}a^{-1})=1.\label{16.A:aBBa}
\end{align}
But $a^{-1}dc^{-1}a>1$ by Lemma \ref{lemma:case16.A:AdCa}, so (\ref{16.A:aBBa}) shows that:
\begin{align*}
a^{-1}b^{2}a^{-1}&<1\\
\Rightarrow{}ab^{-2}a&>1.\qedhere
\end{align*}
\end{proof}

\begin{lemma} In \Case{16}{1}{}, $d^{-2}cb>1$.
\label{lemma:case16.A:DDcb}
\end{lemma}
\begin{proof} By Lemma \ref{lemma:eq8} and Lemma \ref{lemma:eq16}:
\begin{align}
d^{-2}cb=&d^{-2}c(cb^{-1}bc^{-1})b\nonumber{}\\
\Rightarrow{}d^{-2}cb=&(d^{-2}c^{2})b^{-1}(bc^{-1}b)\nonumber{}\\
\Rightarrow{}d^{-2}cb=&(ab^{-1})b^{-1}(ad^{-1}a)\nonumber{}\\
\Rightarrow{}d^{-2}cb=&(ab^{-2}a)(d^{-1}a).
\label{16.A:DDcb}
\end{align}
But $ab^{-2}a>1$ by Lemma \ref{lemma:case16.A:aBBa}, and $d^{-1}a>1$ by Lemma \ref{lemma:case16:Da}, so (\ref{16.A:DDcb}) shows that:
\begin{align*}
(ab^{-2}a)(d^{-1}a)&>1\\
\Rightarrow{}d^{-2}cb&>1.\qedhere
\end{align*}
\end{proof}

\begin{lemma} In \Case{16}{1}{}, $a^{-1}d^{-2}c^{2}>1$.
\label{lemma:case16.A:ADDcc}
\end{lemma}
\begin{proof} By Lemma \ref{lemma:eq8}:
\begin{align}
b^{-1}c^{-2}d^{2}a=1\nonumber{}\\
\Rightarrow{}b^{-1}(aa^{-1})c^{-2}d^{2}a=1\nonumber{}\\
\Rightarrow{}(b^{-1}a)(a^{-1})(c^{-2}d^{2}a)=1.
\label{16.A:ADDcc}
\end{align}
But $b^{-1}a>1$ in general, and $a^{-1}>1$ in case 16, so (\ref{16.A:ADDcc}) shows that:
\begin{align*}
c^{-2}d^{2}a&<1\\
\Rightarrow{}a^{-1}d^{-2}c^{2}&>1.\qedhere
\end{align*}
\end{proof}

\begin{corollary} In \Case{16}{1}{} $a^{-1}d^{-2}c>1$.
\label{corollary:case16.A:ADDc}
\end{corollary}
\begin{proof} We know $c^{-1}>1$ in Case 16, and by Lemma~\ref{lemma:case16.A:ADDcc}, we know $a^{-1}d^{-2}c^{2}>1$ in \Case{16}{1}{}. Therefore,
\begin{align*}
 (a^{-1}d^{-2}c^{2})(c^{-1})&>1\\
 \Rightarrow{}a^{-1}d^{-2}c&>1.\qedhere
 \end{align*}
 \end{proof}
 
\begin{lemma} In \Case{16}{1}{}, $a^{-1}c^{-1}da > 1$.
\label{lemma:case16:A:ACda}
\end{lemma}
\begin{proof} By the third group relation, we have:
\begin{align}
(d^{-1}c)^{10n+3}c^{-1}bc^{-2} =1\nonumber{}\\
\Rightarrow{}(d^{-1}c)^{10n+2}d^{-1}bc^{-2} =1\nonumber{}\\
\Rightarrow{}d^{-1}bc^{-2}(d^{-1}c)^{10n+2}=1\nonumber{}\\
\Rightarrow{}(a^{-1}a)d^{-1}b(a^{-1}ad^{-1}d)c^{-2}(aa^{-1})(d^{-1}c)^{10n+2}(aa^{-1})=1\nonumber{}\\
\Rightarrow{}(a^{-1})(ad^{-1})(ba^{-1})(ad^{-1})(dc^{-1})(c^{-1}a)(a^{-1}d^{-1}ca)^{10n+2}(a^{-1})=1.
\label{16.A:ACda}
\end{align} 
But $a^{-1}>1$ in Case 16, $ad^{-1}>1$ in \Case{16}{1}{}, $ba^{-1}>1$ by Lemma \ref{lemma:case16:bA}, $dc^{-1}>1$ in \Case{16}{1}{}, and $c^{-1}a>1$ by Lemma \ref{lemma:case16:Ca}, so (\ref{16.A:ACda}) shows that:
\begin{align*}
a^{-1}d^{-1}ca&<1\\
\Rightarrow{}a^{-1}c^{-1}da&>1.\qedhere
\end{align*}
\end{proof}

\begin{lemma} In \Case{16}{1}{}, $ac^{-1}da^{-1}>1$.
\label{lemma:case16.A:aCdA}
\end{lemma}
\begin{proof} By the third group relation, we have:
\begin{align}
(d^{-1}c)^{10n+3}c^{-1}bc^{-2}=1\nonumber{}\\
\Rightarrow{}(a^{-1}a)(d^{-1}c)^{10n+2}(a^{-1}a)(d^{-1}c)c^{-1}b(a^{-1}a)c^{-1}(a^{-1}a)c^{-1}=1\nonumber{}\\
\Rightarrow{}(a^{-1})(ad^{-1}ca^{-1})^{10n+2}(ad^{-1})(cc^{-1})(ba^{-1})(ac^{-1})(a^{-1})(ac^{-1})=1\nonumber{}\\
\Rightarrow{}(a^{-1})(ad^{-1}ca^{-1})^{10n+2}(ad^{-1})(ba^{-1})(ac^{-1})(a^{-1})(ac^{-1})=1.
\label{16.A:aCdA}
\end{align}
But $ba^{-1}>1$ by Lemma \ref{lemma:case16:bA}, $a^{-1}>1$ in Case 16, $ad^{-1}>1$ in \Case{16}{1}{}, , and $ac^{-1}>1$ in \Case{16}{1}{}, so (\ref{16.A:aCdA}) shows that:
\begin{align*}
ad^{-1}ca^{-1}&<1\\
\Rightarrow{}ac^{-1}da^{-1}&>1.\qedhere
\end{align*}
\end{proof}

\begin{lemma} In \Case{16}{1}{}, $b^{-1}cd^{-1}a>1$.
\label{lemma:case16.A:BcDa}
\end{lemma}
\begin{proof} By Lemma \ref{lemma:eq5}, we have:
\begin{align}
dc^{-1}b^{2}a^{-2}=1\nonumber{}\\
\Rightarrow{}a^{-1}dc^{-1}b^{2}a^{-1}=1\nonumber{}\\
\Rightarrow{}(a^{-1}dc^{-1}b)(ba^{-1})=1.
\label{16.A:BcDa}
\end{align}
But $ba^{-1}>1$ by Lemma \ref{lemma:case16:bA}, so (\ref{16.A:BcDa}) shows that:
\begin{align*}
a^{-1}dc^{-1}b&<1\\
\Rightarrow{}b^{-1}cd^{-1}a&>1.\qedhere
\end{align*}
\end{proof}

\begin{lemma} In \Case{16}{1}{}, $ad^{-2}c>1$.
\label{lemma:16:A:aDDc}
\end{lemma}
\begin{proof} By Lemma~\ref{lemma:eq7}, we have:
\begin{align}
d^{2}a^{-1}d&=c^{2}b^{-1}c\nonumber{}\\
\Rightarrow{}bc^{-2}d^{2}a^{-1}dc^{-1}&=1\nonumber{}\\
\Rightarrow{}b(a^{-1}a)(d^{-1}d)c^{-2}d^{2}a^{-1}dc^{-1}&=1\nonumber{}\\
\Rightarrow{}(ba^{-1})(ad^{-1})(dc^{-1})(c^{-1}d^{2}a^{-1})(dc^{-1})&=1.\label{lemma:16:A:aDDc:contradiction}
\end{align}
But $ba^{-1}>1$ in Case 16 by Lemma~\ref{lemma:case16:bA}, $ad^{-1}>1$ in \Case{16}{1}{} by assumption, and $dc^{-1}>1$ in \Case{16}{1}{} by assumption. Therefore, (\ref{lemma:16:A:aDDc:contradiction}) shows that:
\begin{align*}
c^{-1}d^{2}a^{-1}&<1\\
\Rightarrow{}ad^{-2}c&>1.\qedhere
\end{align*}
\end{proof}

\begin{lemma} In \Case{16}{1}{}, $a^{-1}d^{-1}cd^{-1}c>1$.
\label{lemma:16:A:ADcDc}
\end{lemma}
\begin{proof} By Lemma~\ref{lemma:eq7}, we have:
\begin{align}
d^{2}a^{-1}d&=c^{2}b^{-1}c\nonumber{}\\
\Rightarrow{}bc^{-2}d^{2}a^{-1}dc^{-1}&=1\nonumber{}\\
\Rightarrow{}b(a^{-1}a)c^{-1}(da^{-1}ad^{-1})(d^{-1}cc^{-1}d)c^{-1}d(aa^{-1})da^{-1}dc^{-1}&=1\nonumber{}\\
\Rightarrow{}(ba^{-1})(ac^{-1}da^{-1})(ad^{-2}c)(c^{-1}dc^{-1}da)(a^{-1}da^{-1})(dc^{-1})&=1.\label{lemma:16:A:ADcDc:contradiction}
\end{align}
But $ba^{-1}>1$ in Case 16 by Lemma~\ref{lemma:case16:bA}, $ac^{-1}da^{-1}>1$ by Lemma~\ref{lemma:case16.A:aCdA}, $ad^{-2}c>1$ in \Case{16}{1}{} by Lemma~\ref{lemma:16:A:aDDc}, $a^{-1}da^{-1}>1$ in Case 16 since $a^{-1}da^{-1}=(ba^{-1})^{10n}$ (see (\ref{eq1:3})) which is positive by Lemma~\ref{lemma:case16:bA}, and $dc^{-1}>1$ in \Case{16}{1}{} by assumption. Therefore, (\ref{lemma:16:A:ADcDc:contradiction}) shows that:
\begin{align*}
c^{-1}dc^{-1}da&<1\\
\Rightarrow{}a^{-1}d^{-1}cd^{-1}c&>1.\qedhere
\end{align*}
\end{proof}

\begin{lemma} In \Case{16}{1}{}, $a^{-1}c^{-2}dca>1$.
\label{lemma:case16:A:ACCdca}
\end{lemma}
\begin{proof} Suppose that $G_n$ is left-orderable, and suppose (for contradiction), that $a<1$, $b<1$, $c<1$, and $d<1$. By the third group relation, we have: 
\begin{align}
(d^{-1}c)^{10n+3}c^{-1}bc^{-2}&=1\nonumber{}\\
\Rightarrow{}(d^{-1}c)^{10n+2}d^{-1}bc^{-2}&=1\nonumber{}\\
\Rightarrow{}d^{-1}bc^{-2}(d^{-1}c)^{10n+2}&=1\nonumber{}\\
\Rightarrow{}(d^{-1}bc^{-2})(c^{-1}aa^{-1}c^{-1})(d^{-1}c)^{10n-2}(caa^{-1}c^{-1})&\nonumber{}\\
(daa^{-1}d^{-1})(d^{-1}c)(cc^{-1}dd^{-1})(d^{-1}c)(bb^{-1}aa^{-1})(d^{-1}cd^{-1}c)&=1\nonumber{}\\
\Rightarrow{}(d^{-1})(bc^{-2})(c^{-1}a)(a^{-1}c^{-1}d^{-1}cca)^{10n-2}(a^{-1}c^{-1}da)&\nonumber{}\\
(a^{-1}d^{-1}d^{-1}cc)(c^{-1}d)(d^{-1}d^{-1}cb)(b^{-1}a)(a^{-1}d^{-1}cd^{-1}c)&=1\nonumber{}\\
\Rightarrow{}(d^{-1})(bc^{-2})(c^{-1}a)(a^{-1}c^{-1}d^{-1}c^{2}a)^{10n-2}(a^{-1}c^{-1}da)&\nonumber{}\\
(a^{-1}d^{-2}c^{2})(c^{-1}d)(d^{-2}cb)(b^{-1}a)(a^{-1}d^{-1}cd^{-1}c)&=1.
\label{16.A:ACCdca}
\end{align}
But $d^{-1}>1$ by Case 16, $bc^{-2}>1$ by \Case{16}{1}{}, $c^{-1}a>1$ by Lemma \ref{lemma:case16:Ca}, $a^{-1}c^{-1}da>1$ by Lemma \ref{lemma:case16:A:ACda}, $a^{-1}d^{-2}c^{2}>1$ by Lemma \ref{lemma:case16.A:ADDcc}, $c^{-1}d>1$ by Lemma \ref{lemma:case16:Cd}, $d^{-2}cb>1$ by Lemma \ref{lemma:case16.A:DDcb}, $b^{-1}a>1$ by general assumption, and $a^{-1}d^{-1}cd^{-1}c>1$ by Lemma \ref{lemma:16:A:ADcDc}, so (\ref{16.A:ACCdca}) shows that:
\begin{align*}
a^{-1}c^{-1}d^{-1}c^{2}a&<1\\
\Rightarrow{}a^{-1}c^{-2}dca&>1.\qedhere
\end{align*}
\end{proof}

\begin{proposition}  If $G_n$ is left-orderable, then \Case{16}{1}{} ($a<1$, $b<1$, $c<1$, and $d<1$) is impossible.
\label{proposition:case16.A}
\end{proposition}
\begin{proof} By Lemma \ref{lemma:eq8}, we have:
\begin{align}
b^{-1}c^{-2}d^{2}a&=1\nonumber{}\\
\Rightarrow{}b^{-1}(aa^{-1})c^{-2}d(caa^{-1}c^{-1})da&=1\nonumber{}\\
\Rightarrow{}(b^{-1}a)(a^{-1}c^{-2}dca)(a^{-1}c^{-1}da)&=1.
\label{proposition:case16.A:contradiction}
\end{align}
But $b^{-1}a>1$ by general assumption, $a^{-1}c^{-2}dca>1$ by Lemma \ref{lemma:case16:A:ACCdca}, and $a^{-1}c^{-1}da>1$ by Lemma \ref{lemma:case16:A:ACda}. Therefore, (\ref{proposition:case16.A:contradiction}) states that a product of positives is the identity, a contradiction.
\end{proof}

\subsection{\Case{16}{2}{}}

\noindent{} As a reminder, since we are working in a sub-case of Case 16, we know $a<1$, $b<1$, $c<1$, and $d<1$.

\begin{lemma} In \Case{16}{2}{}, $a^{-1}bc^{-1}b>1$.
\label{lemma:case16:AbCb}
\end{lemma}
\begin{proof}By Lemma~\ref{lemma:eq16}, we have:
\begin{align*}
bc^{-1}b=ad^{-1}a.
\end{align*}
Thus:
\begin{align*}
a^{-1}(bc^{-1}b)=a^{-1}(ad^{-1}a)=(d^{-1}a).
\end{align*}
But $d^{-1}a>1$ in Case 16 by Lemma~\ref{lemma:case16:Da}; therefore, we know:
\begin{align*}
a^{-1}bc^{-1}b=d^{-1}a&>1.\qedhere
\end{align*}
\end{proof}

\begin{corollary} In \Case{16}{2}{}, $a^{-1}bc^{-1}a>1$ and $a^{-1}bc^{-1}>1$.
\label{corollary:case16:AbC}
\label{corollary:case16:AbCa}
\end{corollary}
\begin{proof}By Lemma \ref{lemma:case16:AbCb}, we have
\begin{align*}
a^{-1}bc^{-1}b>1.
\end{align*}
Since $b^{-1}a>1$, we have:
\begin{align*}
b^{-1}a&>1\\
\Rightarrow{}a^{-1}bc^{-1}b(b^{-1}a)>a^{-1}bc^{-1}b&>1\\
\Rightarrow{}a^{-1}bc^{-1}a>a^{-1}bc^{-1}b&>1.
\end{align*}
By $b^{-1}>1$, we have:
\begin{align*}
b^{-1}&>1\\
\Rightarrow{}a^{-1}bc^{-1}b(b^{-1})>a^{-1}bc^{-1}b&>1\\
\Rightarrow{}a^{-1}bc^{-1}>a^{-1}bc^{-1}b&>1.\qedhere
\end{align*}
\end{proof}

\begin{lemma} In \Case{16}{2}{}, $c^{-1}d^{-1}c > 1 $.
\label{lemma:case16:CDc}
\end{lemma}
\begin{proof} By Lemma~\ref{lemma:eq8}, we have:
\begin{align}
b^{-1}c^{-2}d^{2}a &= 1\nonumber{}\\
\Rightarrow{}b^{-1}(b^{-1}b)c^{-2}d(cc^{-1})d(cc^{-1})a &= 1\nonumber{}\\
\Rightarrow{}b^{-1}b^{-1}(aa^{-1})bc^{-2}d(cc^{-1})d(cc^{-1})a &= 1\nonumber{}\\
\Rightarrow{}(b^{-1})(b^{-1}a)(a^{-1}bc^{-1})(c^{-1}dc)(c^{-1}dc)(c^{-1}a) &=1.\label{lemma:case16:CDc:contradiction}
\end{align}
Now $b^{-1}>1$ and $b^{-1}a>1$ by general assumption, $a^{-1}bc^{-1}>1$ in \Case{16}{2}{} by Corollary~\ref{corollary:case16:AbC}, and $c^{-1}a>1$ in Case 16 by Corollary~\ref{lemma:case16:Ca}. Therefore, (\ref{lemma:case16:CDc:contradiction}) shows that
\begin{align*}
c^{-1}dc&<1\\
\Rightarrow{}c^{-1}d^{-1}c&>1.\qedhere
\end{align*}
\end{proof}

\begin{lemma} In \Case{16}{2}{}, $b^{-1}d^{-1}b>1$.
\label{lemma:case16:BDb}
\end{lemma}
\begin{proof} By Lemma~\ref{lemma:eq8}, we have:
\begin{align}
d^{2}a&=c^{2}b\nonumber{}\\
\Rightarrow{}b^{-1}c^{-2}d^{2}a&=1\nonumber{}\\
\Rightarrow{}b^{-1}c^{-2}(bb^{-1})d(bb^{-1})d(bb^{-1})a&=1\nonumber{}\\
\Rightarrow{}(b^{-1})(c^{-1})(c^{-1}b)(b^{-1}db)(b^{-1}db)(b^{-1}a)&=1.\label{lemma:case16:BDb:contradiction}
\end{align}
Now $b^{-1}$, $c^{-1}$, and $b^{-1}a$ are positive by assumption in Case 16, and $c^{-1}b>1$ by Lemma~\ref{lemma:case16:Cb}. Therefore, (\ref{lemma:case16:BDb:contradiction}) shows that
\begin{align*}
b^{-1}db&<1\\
\Rightarrow{}b^{-1}d^{-1}b&>1.\qedhere
\end{align*}
\end{proof}

\begin{lemma} In \Case{16}{2}{}, $a^{-1}dc^{-1}a>1$.
\label{lemma:case16:AdCa}
\end{lemma}

\begin{proof} Starting from the third group relation, we have:
\begin{align}
(d^{-1}c)^{10n+3}c^{-1}bc^{-2}&=1\nonumber{}\\
\Rightarrow{}c^{-1}(cd^{-1})^{10n+3}bc^{-2}&=1\nonumber{}\\
\Rightarrow{}(c^{-1}a)(a^{-1}cd^{-1}a)^{10n+3}(a^{-1}bc^{-1})(c^{-1})&=1.\label{lemma:case16:AdCa:contradiction}
\end{align}
Now $c^{-1}>1$ in Case 16 by assumption, $c^{-1}a>1$ in Case 16 by Corollary~\ref{lemma:case16:Ca}, and $a^{-1}bc^{-1}>1$ in \Case{16}{2}{} by Corollary~\ref{corollary:case16:AbC}. Therefore, (\ref{lemma:case16:AdCa:contradiction}) shows that:
\begin{align*}
a^{-1}cd^{-1}a&<1\\
\Rightarrow{}a^{-1}dc^{-1}a&>1.\qedhere
\end{align*}
\end{proof}

\begin{lemma} In \Case{16}{2}{}, $d^{-2}cb>1$.
\label{lemma:case16:DDcb}
\end{lemma}

\begin{proof} By Lemma~\ref{lemma:eq8}, we have:
\begin{align}
c^{2}b&=d^{2}a\nonumber{}\\
\Rightarrow{}d^{-2}c^{2}&=ab^{-1}.\label{eq8:4}
\end{align}
By Corollary~\ref{corollary:eq6}, we have:
\begin{align}
a^{2}b^{-2}&=dc^{-1}\nonumber{}\\
\Rightarrow{}a^{2}b^{-2}cd^{-1}&=1.\label{eq6:4}
\end{align}
Combining (\ref{eq8:4}) and (\ref{eq6:4}), we find:
\begin{align}
ad^{-2}c^{2}b^{-1}cd^{-1}&=1\label{eq18}\\
\Rightarrow{}dc^{-1}bc^{-2}d^{2}a^{-1}&=1\nonumber{}\\
\Rightarrow{}a^{-1}dc^{-1}bc^{-2}d^{2}&=1\nonumber{}\\
\Rightarrow{}a^{-1}dc^{-1}(aa^{-1})bc^{-1}(bb^{-1})c^{-1}d^{2}&=1\nonumber{}\\
\Rightarrow{}(a^{-1}dc^{-1}a)(a^{-1}bc^{-1}b)(b^{-1}c^{-1}d^{2})&=1.\label{lemma:case16:DDcb:contradiction}
\end{align}
Now $a^{-1}dc^{-1}a>1$ in \Case{16}{2}{} by Lemma~\ref{lemma:case16:AdCa} and $a^{-1}bc^{-1}b>1$ in \Case{16}{2}{} by Lemma~\ref{lemma:case16:AbCb}. Therefore, (\ref{lemma:case16:DDcb:contradiction}) shows that
\begin{align*}
b^{-1}c^{-1}d^{2}&<1\\
\Rightarrow{}d^{-2}cb&>1.\qedhere
\end{align*}
\end{proof}

\begin{lemma} In \Case{16}{2}{}, $dc^{2}a^{-1}>1$.
\label{lemma:case16:dccA}
\end{lemma}
\begin{proof} By the third group relation, we have:
\begin{align*}
(d^{-1}c)^{10n+3}c^{-1}bc^{-2}=1\\
\Rightarrow{}d^{-1}c(d^{-1}c)^{10n+2}c^{-1}b(a^{-1}a)c^{-2}=1\\
\Rightarrow{}(cd^{-1})^{10n+2}(cc^{-1})(ba^{-1})(ac^{-2}d^{-1})=1\\
\Rightarrow{}(cd^{-1})^{10n+2}(ba^{-1})(ac^{-2}d^{-1})=1,
\end{align*}
where the last equality implies $ac^{-2}d^{-1}<1$, since $cd^{-1}>1$ in \Case{16}{2}{} by assumption, and $ba^{-1}>1$ by Lemma~\ref{lemma:case16:bA}. Therefore, we know $dc^{2}a^{-1}>1$.
\end{proof}

\begin{lemma} In \Case{16}{2}{}, $dab^{-1}a^{-1}>1$.
\label{lemma:case16:daBA}
\end{lemma}
\begin{proof} By Lemma~\ref{lemma:eq8}, we have:
\begin{align*}
c^{2}b&=d^{2}a\\
c^{2}ba^{-1}d^{-2}&=1\\
\Rightarrow{}d^{-1}c^{2}ba^{-1}d^{-1}&=1\\
\Rightarrow{}d^{-1}(d^{-1}d)c^{2}(a^{-1}a)ba^{-1}d^{-1}&=1\\
\Rightarrow{}(d^{-2})(dc^{2}a^{-1})(aba^{-1}d^{-1})&=1,
\end{align*}
where the last equality implies $aba^{-1}d^{-1}<1$, since $d^{-1}>1$ in Case 16 by assumption, and $dc^{2}a^{-1}>1$ in \Case{16}{2}{} by Lemma \ref{lemma:case16:dccA}. Therefore, we know $dab^{-1}a^{-1}>1$.
\end{proof}

\begin{lemma} In \Case{16}{2}{}, $b^{-1}cd^{-1}a>1$.
\label{lemma:case16:BcDa}
\end{lemma}
\begin{proof} By Lemma~\ref{lemma:eq5}, we have:
\begin{align*}
d^{-1}a^{2}&=c^{-1}b^{2}\\
\Rightarrow{}dc^{-1}b^{2}a^{-2}&=1\\
\Rightarrow{}a^{-1}dc^{-1}b^{2}a^{-1}&=1\\
\Rightarrow{}(a^{-1}dc^{-1}b)(ba^{-1})&=1,
\end{align*}
where the last equality implies $a^{-1}dc^{-1}b<1$, since $ba^{-1}>1$ in Case 16 by Lemma~\ref{lemma:case16:bA}. Therefore, we know $b^{-1}cd^{-1}a>1$.
\end{proof}

\begin{lemma} In \Case{16}{2}{}, $ab^{-1}d^{-1}a>1$.
\label{lemma:case16:aBDa}
\end{lemma}
\begin{proof} By Lemma~\ref{lemma:eq5}, we have:
\begin{align*}
d^{-1}a^{2}&=c^{-1}b^{2}\\
\Rightarrow{}dc^{-1}b^{2}a^{-2}&=1\\
\Rightarrow{}a^{-1}dc^{-1}(aa^{-1})b(c^{-1}bb^{-1}c)(d^{-1}aa^{-1}d)ba^{-1}&=1\\
\Rightarrow{}(a^{-1}dc^{-1}a)(a^{-1}bc^{-1}b)(b^{-1}cd^{-1}a)(a^{-1}dba^{-1})&=1,
\end{align*}
where the last equality implies $a^{-1}dba^{-1}<1$, since $a^{-1}dc^{-1}a>1$ in \Case{16}{2}{} by Lemma \ref{lemma:case16:AdCa}, $a^{-1}bc^{-1}b>1$ in \Case{16}{2}{} by Lemma \ref{lemma:case16:AbCb}, and $b^{-1}cd^{-1}a>1$ in \Case{16}{2}{} by Lemma \ref{lemma:case16:BcDa}. Therefore, we know $ab^{-1}d^{-1}a>1$.
\end{proof}

\begin{lemma} In \Case{16}{2}{}, $cda^{-1}>1$.
\label{lemma:case16:cdA}
\end{lemma}
\begin{proof} In \Case{16}{2}{}, $1>dc^{-1}$ by assumption, and so we have:
\begin{align*}
1>dc^{-1}&=(d^{-1}d)d(da^{-1}ad^{-1})c^{-1}\\
\Rightarrow{}1>dc^{-1}&=(d^{-1})(d^{3}a^{-1})(ad^{-1}c^{-1}),
\end{align*}
where the last equality implies $ad^{-1}c^{-1}<1$, since $d^{-1}>1$ and $d^{3}a^{-1}>1$ in \Case{16}{2}{} by assumption.
\end{proof}

\begin{lemma} In \Case{16}{2}{}, $a^{-1}dc^{-1}>1$.
\label{lemma:case16:AdC}
\end{lemma}
\begin{proof}
Starting from the third group relation, we have:
\begin{align}
(d^{-1}c)^{10n+3}c^{-1}bc^{-2}&=1\nonumber{}\\
\Rightarrow{}c^{-1}(cd^{-1})^{10n+3}bc^{-2}&=1\nonumber{}\\
\Rightarrow{}c^{-1}(cd^{-1})^{10n+2}(cd^{-1})(aa^{-1})(bc^{-1})c^{-1}&=1\nonumber{}\\
\Rightarrow{}c^{-1}(cd^{-1})^{10n+2}(cd^{-1}a)(a^{-1}bc^{-1})c^{-1}&=1.\label{lemma:case16:AdC:contradiction}
\end{align}
Now $c^{-1}>1$ in Case 16, $cd^{-1}>1$ in \Case{16}{2}{} by assumption, and $a^{-1}bc^{-1}>1$ in \Case{16}{2}{} by Corollary~\ref{corollary:case16:AbC}. Therefore, (\ref{lemma:case16:AdC:contradiction}) shows that:
\begin{align*}
cd^{-1}a&<1\\
\Rightarrow{}a^{-1}dc^{-1}&>1.\qedhere
\end{align*}
\end{proof}

\begin{lemma} In \Case{16}{2}{}, $ca^{-1}bc^{-1}>1$.
\label{lemma:case16:cAbC}
\end{lemma}
\begin{proof} By the first group relation, we have:
\begin{align*}
(b^{-1}a)^{10n}a^{-2}d&=1\\
\Rightarrow{}a^{-2}d(b^{-1}a)^{10n}&=1\\
\Rightarrow{}a^{-2}d(c^{-1}c)(b^{-1}a)^{10n-1}(c^{-1}c)(b^{-1}a)&=1\\
\Rightarrow{}(a^{-1}dc^{-1})(cb^{-1}ac^{-1})^{10n-1}(cb^{-1})&=1,
\end{align*}
where the last equality implies $cb^{-1}ac^{-1}<1$, since $a^{-1}dc^{-1}>1$ in \Case{16}{2}{} by Lemma \ref{lemma:case16:AdC}, and $cb^{-1}>1$ in \Case{16}{2}{} by assumption. Therefore, we know $ca^{-1}bc^{-1}>1$.
\end{proof}

\begin{lemma} In \Case{16}{2}{}, $b^{-1}c^{-1}db>1$.
\label{lemma:case16:BCdb}
\end{lemma}
\begin{proof} Starting from the third group relation, we have:
\begin{align}
(d^{-1}c)^{10n+3}c^{-1}bc^{-2}&=1\nonumber{}\\
\Rightarrow{}c^{-2}(d^{-1}c)^{10n+3}c^{-1}b&=1\nonumber{}\\
\Rightarrow{}c^{-2}(d^{-1}c)^{10n+2}d^{-1}b&=1\nonumber{}\\
\Rightarrow{}(c^{-1})(c^{-1}b)(b^{-1}d^{-1}cb)^{10n+2}(b^{-1}d^{-1}b)&=1.\label{lemma:case16:BCdb:contradiction}
\end{align}
Now $c^{-1}>1$ in Case 16 by assumption, $c^{-1}b>1$ in Case 16 by Lemma~\ref{lemma:case16:Cb}, and $b^{-1}d^{-1}b>1$ in \Case{16}{2}{} by Lemma~\ref{lemma:case16:BDb}. Therefore, (\ref{lemma:case16:BCdb:contradiction}) shows that:
\begin{align*}
b^{-1}d^{-1}cb&<1\\
\Rightarrow{}b^{-1}c^{-1}db&>1.\qedhere
\end{align*}
\end{proof}

\begin{lemma} In \Case{16}{2}{}, $c^{-2}dc>1$.
\label{lemma:case16:CCdc}
\end{lemma}
\begin{proof}
Starting from the third group relation, we have:
\begin{align}
(d^{-1}c)^{10n+3}c^{-1}bc^{-2}&=1\nonumber{}\\
\Rightarrow{}c^{-1}bc^{-2}(d^{-1}c)^{10n+3}&=1\nonumber{}\\
\Rightarrow{}c^{-1}bc^{-2}(d^{-1}c)^{10n}(d^{-1}c)^{3}&=1\nonumber{}\\
\Rightarrow{}(d^{-1}c)^{2}c^{-1}bc^{-2}(d^{-1}c)^{10n}(d^{-1}c)&=1\nonumber{}\\
\Rightarrow{}d^{-1}cd^{-1}bc^{-2}(d^{-1}c)^{10n}(d^{-1}c)&=1\nonumber{}\\
\Rightarrow{}d^{-1}cd^{-1}bc^{-1}(c^{-1}d^{-1}c^{2})^{10n}(c^{-1}d^{-1}c)&=1\nonumber{}\\
\Rightarrow{}d^{-1}(d^{-1}cbb^{-1}c^{-1}d)(bb^{-1})cd^{-1}(aa^{-1})bc^{-1}(c^{-1}d^{-1}c^{2})^{10n}(c^{-1}d^{-1}c)&=1\nonumber{}\\
\Rightarrow{}(d^{-2}cb)(b^{-1}c^{-1}db)(b^{-1}cd^{-1}a)(a^{-1}bc^{-1})(c^{-1}d^{-1}c^{2})^{10n}(c^{-1}d^{-1}c)&=1.\label{lemma:case16:CCdc:contradiction}
\end{align}
Now $d^{-2}cb>1$ in \Case{16}{2}{} by Lemma~\ref{lemma:case16:DDcb}, $b^{-1}c^{-1}db>1$ in \Case{16}{2}{} by Lemma~\ref{lemma:case16:BCdb}, $b^{-1}cd^{-1}a>1$ in \Case{16}{2}{} by Lemma~\ref{lemma:case16:BcDa}, $a^{-1}bc^{-1}>1$ in \Case{16}{2}{} by Corollary~\ref{corollary:case16:AbC}, and $c^{-1}d^{-1}c>1$ in \Case{16}{2}{} by Lemma~\ref{lemma:case16:CDc}. Therefore, (\ref{lemma:case16:CCdc:contradiction}) shows that
\begin{align*}
c^{-1}d^{-1}c^{2}&<1\\
\Rightarrow{}c^{-2}dc&>1.\qedhere
\end{align*}
\end{proof}

\begin{lemma} In \Case{16}{2}{}, $cab^{-1}c^{-1}>1$.
\label{lemma:case16:caBC}
\end{lemma}
\begin{proof}
By Lemma~\ref{lemma:eq8} we have:
\begin{align}
d^{2}a&=c^{2}b\nonumber{}\\
\Rightarrow{}d^{2}&=c^{2}ba^{-1}\nonumber{}\\
\Rightarrow{}d^{3}&=dc^{2}ba^{-1}.\label{eq8:5}
\end{align}
By Lemma~\ref{lemma:eq9} we have:
\begin{align}
c^{-3}d^{3}&=b^{-1}a\nonumber{}\\
\Rightarrow{}d^{3}&=c^{3}b^{-1}a.\label{eq9:2}
\end{align}
Combining (\ref{eq8:5}) and (\ref{eq9:2}), we find:
\begin{align}
c^{3}b^{-1}a&=dc^{2}ba^{-1}\nonumber{}\\
\Rightarrow{}c^{-1}d^{-1}c^{3}b^{-1}a^{2}b^{-1}c^{-1}&=1\nonumber{}\\
\Rightarrow{}cba^{-2}bc^{-3}dc&=1\label{eq12}\\
\Rightarrow{}cba^{-1}(c^{-1}c)a^{-1}bc^{-1}c^{-2}dc&=1\nonumber{}\\
\Rightarrow{}(cba^{-1}c^{-1})(ca^{-1}bc^{-1})(c^{-2}dc)&=1.\label{lemma:case16:caBC:contradiction}
\end{align}
Now $ca^{-1}bc^{-1}>1$ in \Case{16}{2}{} by Lemma~\ref{lemma:case16:cAbC} and $c^{-2}dc>1$ in \Case{16}{2}{} by Lemma~\ref{lemma:case16:CCdc}. Therefore, (\ref{lemma:case16:caBC:contradiction}) shows that:
\begin{align*}
cba^{-1}c^{-1}&<1\\
\Rightarrow{}cab^{-1}c^{-1}&>1.\qedhere
\end{align*}
\end{proof}

\begin{proposition} If $G_n$ is left-orderable, then \Case{16}{2}{} is impossible.
\label{proposition:case16.B}
\end{proposition}
\begin{proof} Suppose $G_n$ is left-orderable, and suppose (for contradiction) that the signs of elements are as in \Case{16}{2}{}. Starting from the first group relation, we have:
\begin{align*}
(a^{-1}b)^{10n}d^{-1}a^{2}&=1\\
\Rightarrow{}a^{-2}d(b^{-1}a)^{10n}&=1\\
\Rightarrow{}a^{-2}da^{-1}(ab^{-1})^{10n}a&=1\\
\Rightarrow{}da^{-1}(ab^{-1})^{10n}a^{-1}&=1\\
\Rightarrow{}db^{-1}(ab^{-1})^{10n-1}a^{-1}&=1\\
\Rightarrow{}db^{-1}(ab^{-1})^{10n-2}ab^{-1}a^{-1}&=1\\
\Rightarrow{}db^{-1}c^{-1}(cab^{-1}c^{-1})^{10n-2}cab^{-1}a^{-1}&=1\\
\Rightarrow{}d(a^{-1}a)b^{-1}(d^{-1}aa^{-1}d)c^{-1}(cab^{-1}c^{-1})^{10n-2}c(d^{-1}d)ab^{-1}a^{-1}&=1\\
\Rightarrow{}(da^{-1})(ab^{-1}d^{-1}a)(a^{-1}dc^{-1})(cab^{-1}c^{-1})^{10n-2}(cd^{-1})(dab^{-1}a^{-1})&=1.
\end{align*}
This is a contradiction, since $da^{-1}>1$ in \Case{16}{2}{} by assumption, $ab^{-1}d^{-1}a>1$ in \Case{16}{2}{} by Lemma~\ref{lemma:case16:aBDa}, $a^{-1}dc^{-1}>1$ in \Case{16}{2}{} by Lemma~\ref{lemma:case16:AdC}, $cab^{-1}c^{-1}>1$ in \Case{16}{2}{} by Lemma~\ref{lemma:case16:caBC}, $cd^{-1}>1$ in \Case{16}{2}{} by assumption, and $dab^{-1}a^{-1}>1$ in \Case{16}{2}{} by Lemma~\ref{lemma:case16:daBA}.
\end{proof}

%% file: bibliography_text.tex
$\;$\\
\noindent{}{\sc{Department of Mathematics, Columbia University, 2990 Broadway, New York, NY 10027.}}\\
\emph{E-mail address:} {\texttt{fsd2108@columbia.edu}}
\\\\
\noindent{}{\sc{Department of Mathematics, Columbia University, 2990 Broadway, New York, NY 10027.}}\\
\emph{E-mail address:} {\texttt{mvj2110@columbia.edu}}
\\\\
\noindent{}{\sc{Department of Mathematics, Columbia University, 2990 Broadway, New York, NY 10027.}}\\
\emph{E-mail address:} {\texttt{jar2262@columbia.edu}}
\\\\
\noindent{}{\sc{Department of Mathematics, Barnard College, 333 Milbank Hall, New York, NY 10027.}}\\
\emph{E-mail address:} {\texttt{hlz2103@barnard.edu}}